\documentclass[a4paper,twoside,10pt]{article}

\usepackage[a4paper,left=3cm,right=3cm, top=3cm, bottom=3cm]{geometry}
\usepackage[latin1]{inputenc}
\usepackage{mathrsfs}
\usepackage{graphicx}
\usepackage{epstopdf}
\usepackage{amsmath}
\usepackage{amsthm}
\usepackage{amssymb}
\usepackage{hyperref}
\usepackage{stmaryrd}
\usepackage{subfigure}
\usepackage{float}
\usepackage{bigints}
\usepackage{cite}
\usepackage{color}
\usepackage[abs]{overpic}
\usepackage[font=footnotesize,labelfont=bf]{caption}
\usepackage{cases}
\usepackage{tikz}
\usepackage{algorithmic}
\usepackage{rotating}
\usetikzlibrary{arrows}
\usetikzlibrary{quotes,angles}
\usepackage[author={Lorenzo}]{pdfcomment}
\usepackage{soul,xcolor}

\restylefloat{table}
\theoremstyle{plain}
\newtheorem{thm}{Theorem}[section]

\newtheorem{lem}[thm]{Lemma}

\theoremstyle{definition}

\newtheorem{exmp}{Example}[section]
\theoremstyle{remark}
\newtheorem{remark}{Remark}

\newcommand{\D}{D} 
\renewcommand{\O}{\Omega} 
\newcommand{\boldalpha}{\boldsymbol \alpha} 
\newcommand{\boldbeta}{\boldsymbol \beta} 
\renewcommand{\u}{u} 
\newcommand{\uhat}{\widehat \u} 
\newcommand{\vvv}{v} 
\newcommand{\un}{\u_n} 
\newcommand{\unn}{\u} 
\newcommand{\vn}{\vvv_n} 
\newcommand{\vnn}{\vvv} 
\newcommand{\g}{g} 
\newcommand{\gtilde}{\widetilde g} 
\newcommand{\xbold}{\textbf x} 
\newcommand{\Abold}{\textbf A} 
\newcommand{\Babuska}{Babu{\v{s}}ka } 
\renewcommand{\c}{c} 
\DeclareMathSymbol{\cacca}{\mathalpha}{operators}{"4F}
\newcommand{\Omicron}{\mathcal \cacca} 
\newcommand{\V}{V} 
\newcommand{\Vg}{\V_\g} 
\newcommand{\Vzero}{\V_0} 
\newcommand{\Vn}{\V_n} 
\newcommand{\VE}{\V^{\Delta}(\E)} 
\newcommand{\Vng}{\V_{n,\g}} 
\newcommand{\Vnzero}{\V_{n,0}} 
\newcommand{\taun}{\mathcal T _ n} 
\newcommand{\tautilden}{\widetilde {\mathcal T}} 
\newcommand{\E}{K} 
\newcommand{\Ehat}{\widehat \E} 
\newcommand{\T}{T} 
\newcommand{\e}{s} 
\newcommand{\p}{p} 
\newcommand{\pE}{\p} 
\newcommand{\pEhat}{\p_{\Ehat}} 
\newcommand{\pe}{\p} 
\newcommand{\q}{q} 
\newcommand{\qhat}{\widehat \q} 
\newcommand{\pbold}{\mathbf \p} 
\newcommand{\card}{\text{card}} 
\newcommand{\dof}{\text{dof}} 
\renewcommand{\a}{a} 
\newcommand{\an}{\a_n} 
\newcommand{\aE}{\a^\E} 
\newcommand{\anE}{\a_n^\E} 
\newcommand{\Har}{\mathbb H} 
\newcommand{\Pinabla}{\Pi ^{\nabla}} 
\newcommand{\PinablaE}{\Pi ^{\nabla,\E}} 
\renewcommand{\S}{S} 
\newcommand{\SE}{\S^\E} 
\newcommand{\upi}{\u_{\pi}} 
\newcommand{\uI}{\u_I} 
\newcommand{\dist}{\text{dist}} 
\newcommand{\s}{m} 
\newcommand{\sE}{\s} 
\newcommand{\uGL}{\u_{\text{GL}}} 
\newcommand{\utildeGL}{\widetilde\u_{\text{GL}}} 
\newcommand{\zerobold}{\mathbf 0} 
\newcommand{\h}{h} 
\newcommand{\n}{\mathbf n} 
\newcommand{\hE}{\h_\E} 
\newcommand{\NeE}{N_\e^\E} 

\title{The harmonic virtual element method: stabilization and exponential convergence for the Laplace problem on polygonal domains}
\date{}

\begin{tiny}
\author{
\normalsize{
A. Chernov
\thanks{Inst. f\"ur Mathematik, C. von Ossietzky Universit\"at Oldenburg, E-mail: {\tt alexey.chernov@uni-oldenburg.de}},
L. Mascotto
\thanks{Fakult\"at f{\"u}r Mathematik,  Universit\"at Wien, E-mail: {\tt lorenzo.mascotto@unimi.it}}
}}
\end{tiny}

\begin{document}
\maketitle

\begin{abstract}
{We introduce the harmonic virtual element method (harmonic VEM), a modification of the virtual element method (VEM) \cite{VEMvolley} for the approximation of the 2D Laplace equation using polygonal meshes.
The main difference between the harmonic VEM and the VEM is that in the former method only boundary degrees of freedom are employed.
Such degrees of freedom suffice for the construction of a proper energy projector on (piecewise harmonic) polynomial spaces.
The harmonic VEM can also be regarded as an ``$H^1$-conformisation'' of the Trefftz discontinuous Galerkin-finite element method (TDG-FEM) \cite{hmps_harmonicpolynomialsapproximationandTrefftzhpdgFEM}.
We address the stabilization of the proposed method and develop an $\h\p$ version of harmonic VEM for the Laplace equation on polygonal domains.
As in Trefftz DG-FEM, the asymptotic convergence rate of harmonic VEM is exponential and reaches order $\mathcal O ( \exp(-b\sqrt[2]{N}))$, where $N$ is the number of degrees of freedom.
This result overperformes its counterparts in the framework of $\h\p$ FEM \cite{SchwabpandhpFEM} and $\h\p$ VEM \cite{hpVEMcorner}, where the asymptotic rate of convergence is of order $\mathcal O ( \exp(-b\sqrt[3]{N})  )$.}
{virtual element method; polygonal meshes; $\h\p$ Galerkin methods; Trefftz methods; Laplace equation; harmonic polynomials.}
\end{abstract}

\section {Introduction} \label{section introduction}
 {\uccode`A=97
    \uppercase{\typeout{== A ==}}}
In this work, we deal with the approximation of the Laplace equation on polygonal domains based on a novel method, whose main advantage is the fact of having a very small number of degrees of freedom if compared to standard finite element methods (FEM).
This is not of course the first attempt to approximate a Laplace problem with methods based on approximation spaces having small dimension.
Among the other methods available, we limit ourselves to recall only two of them.

The first one is the boundary element method (BEM) \cite{SauterSchwab_BEMbook}. BE spaces consist of functions defined \emph{only} on the boundary of the computational domain. Clearly, the BE space on a boundary mesh of characteristic mesh size $\h$
contains many less degrees of freedom than the corresponding FE space on a volume mesh having the same characteristic mesh size $\h$.
This comes at a price of a fully populated matrix in the resulting system of linear equations, expensive quadrature rules needed for evaluation of matrix entries and expensive numerical reconstruction of the solution in the interior of the computational domain.
These difficulties can be partially alleviated by using advanced \emph{fast boundary element methods} (see e.g. \cite{SauterSchwab_BEMbook} and references therein), that usually results in nontrivial algorithms that are not easy to implement.

A second (and more recent) approach is given by the so-called Trefftz discontinuous Galerkin-FEM (TDG-FEM), which was introduced in \cite{li2012negative, li2006local}
and was generalized to its $\h\p$ version in \cite{hmps_harmonicpolynomialsapproximationandTrefftzhpdgFEM}
(we recall that an $\h\p$ Galerkin method is a method where the convergence of the error is achieved by a proper combination of mesh refinements and an increase of the local degree of accuracy and thereby of the dimension of local spaces).
TDG-FE spaces consist of piecewise harmonic polynomials over a decomposition of the computational domain into triangles and quadrilaterals.
As a consequence, the resulting method has a DG structure, since the dimension of harmonic polynomial spaces is not large enough for enforcing global continuity of the discretization space.
We also point out that in \cite{hmps_harmonicpolynomialsapproximationandTrefftzhpdgFEM} it was provided a result concerning $\h\p$ approximation of harmonic functions
by means of harmonic polynomials, following the ideas of the pioneering works \cite{melenk_phdthesis, melenk1999operator}.

The advantage of TDG-FEM with respect to standard FEM is that the dimension of local spaces considerably reduces still keeping the optimal rate of convergence of the error.
More precisely, for a fixed local polynomial degree $\p$, the dimension of the local TDG-FE space is equal to $2\p+1 \approx 2\p$, whereas the dimension of local FE spaces is $\frac{(\p+1)(\p+2)}{2} \approx \frac{\p^2}{2}$.
This advantage is possible since the degrees of freedom that are removed in TDG-FEM are superfluous  for the approximation of a Laplace equation.
We emphasize that employing piecewise harmonic polynomials leads inevitably to a discontinuous method, which is therefore not anymore $H^1$-conforming.

The approach in \cite{hmps_harmonicpolynomialsapproximationandTrefftzhpdgFEM} can be generalized easily to polygonal TDG-FEM, following e.g. \cite{antonietti2016reviewDG}.
Polygonal methods received an outstanding interest in the last decade by the scientific community due to the high flexibility in dealing with nonstandard geometries.
Among the others, we mention the following methods:
hybrid high--order methods \cite{dipietroErn_hho}, mimetic finite differences \cite{BLM_MFD}, hybrid DG-methods \cite{cockburn_HDG},
polygonal FEM \cite{SukumarTabarraeipolygonalintroduction, talischi2010polygonal, GilletteRandBajaj_generalizedbarycentric}, polygonal DG-FEM \cite{cangianigeorgoulishouston_hpDGFEM_polygon},
BEM-based FEM \cite{Weisser_basic}.

The virtual element method (VEM) is an alternative approach enabling computation of polygonal (polyhedral in 3D) meshes \cite{VEMvolley, hitchhikersguideVEM}.
It is based on globally continuous discretization spaces that generally  consist locally of Trefftz-like functions.
More precisely, the key idea of VEM is that trial and test spaces consists of functions that are solutions to local PDE problems in each element.
Since these local problems do not admit closed form solutions, the bilinear form, and thereby the entries of the stiffness matrix, are not computable in general.
The computable version involves an approximate discrete bilinear form consisting of two additive parts: one that involves local projections on polynomial spaces and a \emph{computable} stabilizing bilinear form.
We emphasize that the approximate discrete bilinear form can be evaluated without explicit knowledge of local basis functions in the interior of the polygonal element: an indirect description via the associated set of internal degrees of freedom suffices.

In \cite{hpVEMbasic}, the $\h\p$ version of VEM for the Poisson problem with quasi-uniform meshes and constant polynomial degree was studied, whereas,
in \cite{hpVEMcorner}, the $\h\p$ version of VEM for the approximation of corner singularities was discussed.
Besides, a multigrid algorithm for the pure $\p$ version of VEM was investigated in \cite{pVEMmultigrid}.
Also, a study regarding the condition number of the stiffness matrix for the $\h\p$ version of 2D and 3D VEM is the topic of \cite{fetishVEM} and \cite{fetishVEM3D}, respectively.

The aim of the present work consists in modifying the $\h\p$ virtual element space of \cite{hpVEMcorner}, trying at the same time to mimick the ``harmonic'' approach of TDG-FEM.
The arising method, which goes under the name of harmonic VEM, makes use only of \emph{boundary} degrees of freedom (the internal degrees of freedom of the standard VEM can be omitted).
More precisely, functions in the harmonic virtual element space are harmonic reconstructions of piecewise continuous polynomial traces over the boundary of the polygons in the polygonal decomposition of the computational domain.
It is immediate to check that the associated space contains (globally discontinuous) piecewise harmonic polynomials.

The stiffness matrix is not computed exactly on the harmonic virtual element space. Its construction is based on two ingredients:
a local energy projector on the space of harmonic polynomials and a stabilizing bilinear form, which only approximates the continuous one and which is computable on the complete space.
As in standard VEM, the projectors and stabilizing bilinear forms are computed only by means  of the degrees of freedom, without the need of knowing trial and test functions in the interior of individual elements explicitly.
Importantly, the implementation of the harmonic VEM does not require two-dimensional quadrature formulas.

The main result of the paper states that, similarly to the $\h\p$ version of TDG-FEM, the asymptotic convergence rate for the energy error is proportional to $\exp (-b \sqrt[2]{N})$, where $N$ is the dimension of the global discretization space.
This result is an improvement of the analogous statement in the framework of the $\h\p$ FEM \cite{SchwabpandhpFEM} and $\h\p$ VEM \cite{hpVEMcorner},
where the rate of decay of the error is proportional to $ \exp (-b \sqrt[3]{N})$. 
As a byproduct of the main result we prove in Sections \ref{subsection stability} and \ref{subsection a p independent stabilization}
novel stabilization estimates that are much sharper than in the general $hp$-VEM \cite{hpVEMcorner} and that are interesting on their own.

We state the difference between the two approaches, namely the TDG-FEM \cite{hmps_harmonicpolynomialsapproximationandTrefftzhpdgFEM} and the harmonic VEM.
The TDG-FEM is a discontinuous method, but local spaces are made of explicitly known functions, i.e. (harmonic) polynomials;
besides, only \emph{internal} degrees of freedom on each element are considered.
The harmonic VEM is a $H^1$-conforming method which only employs \emph{boundary} degrees of freedom;
the basis functions are not known explicitly, but the stiffness matrix can be efficiently built employing only the degrees of freedom.
Both methods are characterized by the fact that the space of piecewise harmonic polynomials is contained in the local approximation spaces;
in fact, the TDG-FE space \emph{is} the space of (globally discontinuous) piecewise harmonic polynomials, while the harmonic virtual element space is richer, in general.

We emphasize that the formulation of the $\h\p$ harmonic VEM presents some improvements with respect to the standard $\h\p$ VEM \cite{hpVEMcorner}.
Less assumptions on the geometry of the polygonal decomposition are required, better bounds for the stabilization are presented and a very tidy result, concerning approximation by functions in the harmonic virtual element space, is proven.

In this paper, we \emph{only} investigate in details the $\h\p$ version of harmonic VEM, that is the method of choice for an efficient
approximation of corner singularities.
A modification of Section \ref{subsection approximation by functions in the Harmonic Element Space}, along with the arguments in \cite{melenk_phdthesis}, leads
to $\h$ approximation results. For the $\p$ version of harmonic VEM, instead, one has to deal with two issues that lie outside the scope of the present paper.

The first one is the pollution effect due to the stabilization of the method, which is typical also of the $\p$ version of VEM, see Lemmata \ref{lemma Strang harmonic} and \ref{lemma stability bounds}, which in fact can be overcome at the price of having a stabilization challenging to be computed, as explained in Section \ref{subsection a p independent stabilization}.

The second one is that the $\p$ approximation estimates by harmonic functions depend on the shape of the domain of approximation via the so-called ``\emph{exterior cone condition}'',
see \cite[Theorem 2]{babumelenk_harmonicpolynomials_approx}. These matters introduce additional technicalities which will not be addressed in this paper.

It is worth mentioning that the theoretical framework of \cite{melenk_phdthesis, MoiolaPhDthesis} for the analysis of methods based on harmonic polynomials can be seen as an intermediate step towards more gruelling challenges.
More precisely, one can use the so-called Vekua theory \cite{vekua1967new} to shift the results related to the approximation of harmonic functions (i.e. for functions belonging to the kernel of the Laplace operator) by means of harmonic polynomials,
to the results related to the approximation of functions in the kernel of more general differential operators by means of \emph{generalized harmonic} polynomials.
A very important example is provided by the approximation of functions in the kernel of the Helmholtz operator; this was investigated in deep in the framework of partition of unity methods (PUM) \cite{melenk_phdthesis} and TDG-FEM \cite{MoiolaPhDthesis}.
An extension of the harmonic VEM to Trefftz-like VEM for the Helmholtz equation has been recently investigated in \cite{ncTVEM_Helmholtz}.

The outline of the paper is the following. In Section \ref{section the model problem}, we present the model problem and we recall some regularity properties of its solution.
In Section \ref{section harmonic virtual element method with nonuniform degrees of accuracy}, we introduce the harmonic VEM; in particular, we discuss the contruction of the stiffness matrix 
along with the construction of a proper stabilization of the method and of an energy projector from local harmonic virtual element spaces into spaces of (piecewise) harmonic polynomials.
After having defined the concept of ``$\h\p$ graded polygonal meshes'', we prove approximation estimates by harmonic polynomials and functions in the harmonic virtual element space in Section \ref{section exponential convergence}.
This approximation scheme results in exponential convergence of the energy error with respect to the total number of degrees of freedom.
Numerical tests validating the theoretical results,
together with a numerical comparison between the performances of the $\h\p$ harmonic VEM and the $\h\p$ VEM, are shown in Section \ref{section numerical results}.

Throughout the paper, we write $f \lesssim g$ for two positive quantities $f$ and $g$ depending on a discretization parameter (typically $h$ or $p$) if there exists a parameter-independent positive constant $c$ such that $f \leq c g$ holds for all values of the parameter. We write $f \approx g$ if $f \lesssim g$ and $g \lesssim f$ holds.

We denote by $\mathbb P_{\p}(\D)$ the spaces of polynomials of degree $\p \in \mathbb N_0$ on a domain $\D$ in one or two variables (depending on the Hausdorff dimension of $\D$).
Finally, we denote by $\Har_\p (\D)$ the space of harmonic polynomials of degree $\p \in \mathbb N_0$ on $\D \subseteq \mathbb R^2$.

\section {The model problem and the functional setting} \label{section the model problem}
Throughout the paper, we will employ the standard notation for Lebesgue and Sobolev spaces on a domain $\D$, see \cite{adamsfournier}. In particular, we denote by $L^2(\D)$ the Lebesgue space of square integrable functions
and by $H^s(\D)$, $s \in \mathbb R_+$, the Sobolev space $W^{2,s}(\D)$.
We set $\Vert \cdot \Vert_{0,\D}$ the standard Lebesgue norm and $\Vert \cdot \Vert_{s,\D}$ and $\vert \cdot \vert_{s,\D}$ the Sobolev norms and seminorms, respectively.

We will use the following notation for partial derivatives:
\begin{equation} \label{notation partial derivatives}
D^{\boldalpha} \u = \partial^{\alpha_1,\alpha_2} \u,\quad \text{where } \boldalpha = (\alpha_1, \alpha_2) \in \mathbb N_0^2.
\end{equation}
We will also write
\begin{equation} \label{derivative from index to multiindices}
\vert D^k \u \vert^2 = \sum_{\boldalpha \in \mathbb N_0^2,\; \vert \boldalpha \vert =k} \vert D^{\boldalpha} \u \vert^2.
\end{equation}
Moreover, we will employ the Sobolev weighted spaces and countably normed spaces defined e.g. in \cite{babuskaguo_curvilinearhpFEM}. For the sake of completeness, we recall their definition.
Given $\O \subseteq \mathbb R^2$ a bounded and simply connected polygonal domain, let $\mathcal V_\O$ be the number of vertices of the closure of $\O$ and let $\{\Abold _i\}_{i=1}^{\mathcal V_\O}$ be the set of such vertices.
We introduce the weight function
\begin{equation} \label{weight function vertices}
\Phi_{\boldbeta} (\xbold) := \prod _{i=1}^{\mathcal V_\O} \min(1,\vert \xbold -\Abold_i \vert)^{\beta_i},
\end{equation}
where $\vert \cdot \vert$ denotes the Euclidean norm and $\boldbeta \in [0,1)^{\mathcal V_\O}$ is the weight vector.
We will write $\Phi_n$, $n \in \mathbb N_0$, meaning that we will consider a weight vector $\boldbeta$ with constant entries $\boldbeta_i=n$, $\forall i=1,\dots, \mathcal V_\O$.
Furthermore, we will denote the particular weight function $\Phi_1$ by $\Phi$.

The weighted Sobolev space $H_{\boldbeta}^{m,\ell}(\O)$, $\boldbeta \in [0,1)^{\mathcal V_\O}$, $m,\ell \in \mathbb N_0$, $m \ge \ell$, are defined as
the completion of $\mathcal C^{\infty}(\overline \O)$ with respect to the norms
\begin{equation} \label{weighted Sobolev norms}
\Vert \u \Vert^2_{H^{m,\ell}_{\boldbeta} (\O)} := \Vert \u \Vert^2_{\ell-1, \O} + \vert \u \vert^2_{H^{m,\ell}_{\beta}(\O)} := \Vert \u \Vert^2_{\ell-1, \O} + \sum_{k=\ell}^m \Vert \Phi_{\boldbeta + k -\ell} \,\vert D^k\u \vert \Vert^2_{0,\O}.
\end{equation}
With an abuse of notation, the sum between the vector $\boldbeta$ and the scalar $k-\ell$ is meant to be
\[
\boldbeta + k - \ell \in \mathbb R^{\mathcal V _ \O},\quad \quad (\boldbeta + k - \ell)_i = \boldbeta_i + k - \ell,\quad i=1,\dots, \mathcal V_\O.
\]
Given $\ell \in \mathbb N_0$ and $\boldbeta \in [0,1)^{\mathcal V _\O}$, we define the countably normed spaces (or \Babuska spaces) as
\begin{equation} \label{countably normed spaces}
\begin{split}
& \mathcal B ^\ell_{\boldbeta} (\O) := \left\{ \u \in H^{m, \ell}_{\boldbeta} (\O)\; \forall m \ge \ell \ge 0 \text{ with }
							\Vert \Phi_ {\boldbeta + k - \ell}\, \vert D^k\u \vert \Vert_{0,\O}\le c_\u d_\u ^{k-\ell} (k-\ell)!\, \forall k \in \mathbb N_0,\, k \ge \ell   \right\},\\
& \Omicron^2_{\boldbeta} (\O) := \left\{ \u \in H^{m,2}_{\boldbeta} (\O)\; \forall m \ge 2  \text{ with } \vert D^k \u (\xbold) \vert  \le c_\u d_\u^k k! \Phi_{\boldbeta + k -1}^{-1} (\xbold)\, \forall k \in \mathbb N_0,\; \forall \xbold \in \overline \O  \right\},\\
\end{split}
\end{equation}
where $c_\u$ and $d_\u$ are two constants greater than or equal to 1, depending \emph{only} on the function $\u$.

We define $\mathcal B_{\boldbeta}^{\frac{3}{2}}(\partial \O)$ and $\Omicron_{\boldbeta}^{\frac{3}{2}}(\partial \E)$
as the set of the traces of functions belonging to $\mathcal B_{\boldbeta}^{2}(\O)$ and $\Omicron_{\boldbeta}^{2}(\O)$, respectively.

From \eqref{derivative from index to multiindices} and \eqref{countably normed spaces}, given $\u \in \Omicron ^2_{\boldbeta} (\O)$, for every $\boldalpha \in \mathbb N_0^2$, $\vert \boldalpha \vert = k \ge 1$, $k \in \mathbb N$, it holds that
\begin{equation} \label{bound derivatives of Babuska functions}
\vert D^{\boldalpha} \u  (\xbold_0)\vert \le \left\vert D^k \u (\xbold_0)\right\vert \le c_\u \frac{d_\u^{|\boldalpha|}}{\Phi_{k}(\xbold_0)} |\boldalpha|!\quad \forall \xbold_0 \in \overline \O,
\end{equation}
since $\beta -1 \in [-1,0)$.

As a consequence, any function in $\Omicron^2_{\boldbeta} (\O)$ admits an analytic continuation on
\begin{equation} \label{analyticity domain}
\mathcal N(\u) := \bigcup _{\xbold_0 \in \overline \O;\, \xbold_0 \neq \Abold_i,\,i=1,\dots,\mathcal V_\O} \left\{  \xbold \in \mathbb R^2 \,\Big|\, \vert \xbold -\xbold_0 \vert < \c \frac{\Phi(\xbold_0)}{d_\u},\;   \forall\, \c \in \left(0,\frac{1}{2}\right) \right\}.
\end{equation}
In order to see this, it suffices to show that the Taylor series
\begin{equation} \label{Taylor series}
\sum_{\boldalpha \in \mathbb N_0^2} \frac{D^{\boldalpha} \u(\xbold _0)}{\boldalpha!} (\xbold - \xbold_0)^{\boldalpha},\quad \xbold_0 \in \overline \O;\, \xbold_0 \ne \Abold_i,\, i=1,\dots, \mathcal V_\O,
\end{equation}
converges uniformly in $\mathcal N (\u)$.

In particular, we prove that it converges uniformly in the ball $B\left( \xbold_0, c\frac{\Phi(\xbold _0)}{d_\u} \right)$ for all $c \in \left(0,\frac{1}{2} \right)$, where $\xbold _0 \in \overline \O$, $\xbold_0 \ne \Abold_i,$ $i=1,\dots, \mathcal V_\O$.
In other words, we have to prove
\[
\sum_{k \in \mathbb N_0} \sum_{\vert\boldalpha \vert =k}  \frac{|D^{\boldalpha} \u(\xbold _0)|}{\boldalpha!} |\xbold - \xbold_0| ^{|\boldalpha|}\le \overline c < \infty,\quad \forall \xbold \in B\left( \xbold_0, c \frac{\Phi(\xbold_0)}{d_\u} \right),
								\quad c \in \left(0,\frac{1}{2}\right),
\]
where $\overline c$ is a positive constant depending only on function $\u$.

Using \eqref{bound derivatives of Babuska functions} and the fact that $\xbold$ belongs to $B\left( \xbold_0, c \frac{\Phi(\xbold_0)}{d_\u} \right)$, we obtain
\begin{equation} \label{computations series}
\begin{split}
& \sum_{k\in \mathbb N_0} \sum_{\vert \boldalpha \vert = k} \frac{\vert D^{\boldalpha} \u (\xbold_0)}{\boldalpha!} \vert \xbold - \xbold _0 \vert^k \le
				\sum_{k\in \mathbb N_0} \sum_{\vert \boldalpha \vert = k} \frac{1}{\boldalpha!} c_\u \frac{d_\u^k}{\Phi_k(\xbold_0)} \vert \boldalpha\vert ! c^k \frac{\Phi_k (\xbold_0)}{d_\u^k(\xbold_0)} \\
& = c_\u \sum_{k\in \mathbb N_0} \sum_{\vert \boldalpha \vert = k}   \frac{\vert \boldalpha \vert !}{\boldalpha !} c^k = c_\u \sum_{k\in \mathbb N_0} \sum_{\ell=0}^k  {{k}\choose {\ell}} c^k
				= c_\u \sum_{k\in \mathbb N_0} \left( 2 c  \right)^k \le \overline c < +\infty,
\end{split}
\end{equation}
since we are assuming that $c \in \left(0,\frac{1}{2} \right)$.


In this paper, we concentrate on the model problem given by the Laplace equation in $\Omega$ endowed with nonhomogeneous Dirichlet boundary conditions.
Given $\g: \partial \Omega \to \mathbb R$, find $u: \Omega \to \mathbb R$ satisfying
\begin{equation} \label{string continuous formulation}
\begin{cases}
\Delta \u = 0 & \text{in } \,\\
\u = \g          & \text{on } \partial \O.\\
\end{cases}
\end{equation}
The weak formulation reads:
\begin{equation} \label{weak continuous problem}
\begin{cases}
\text{find } \u \in \Vg \text{ such that} \\
\a(\u,\vvv) = 0 \quad \forall \vvv \in \Vzero,
\end{cases}
\end{equation}
where
\[
\begin{split}
& \a(\u,\vvv) = (\nabla \u, \nabla \vvv)_{0,\O},\qquad  \V_{\gtilde} := H^1_ {\gtilde}(\O) := \left\{ \vvv \in H^1(\O) \mid \vvv_{|_{\partial \E}} = \gtilde   \right\} \quad \text{for some } \gtilde \in \mathcal B^ {\frac{3}{2}}_{\beta}(\partial \O).\\
\end{split}
\]
It is well known that problem \eqref{weak continuous problem} is well-posed. 

Assuming that the Dirichlet datum $\g \in \mathcal B^{\frac{3}{2}}_{\boldbeta}(\partial \O)$,
then the solution to problem \eqref{weak continuous problem} belongs to~$\mathcal B^2_{\boldbeta}(\O)$ and $\Omicron^2_{\boldbeta} (\O)$ defined in \eqref{countably normed spaces},
see \cite[Theorem 2.2]{babuskaguo_curvilinearhpFEM} and \cite[Theorem 4.44]{SchwabpandhpFEM}.
Owing to \eqref{bound derivatives of Babuska functions} and the subsequent argument, $\u$ is analytic on $\mathcal N(\u)$ defined in \eqref{analyticity domain}.

Before concluding this section, we make the following simplifying assumption:
\begin{equation} \label{assumption one singular vertex}
\begin{cases}
\mathbf 0 \text{ is  a vertex of } \O,\\
\u \text{, solution to \eqref{weak continuous problem}, has only a singularity, precisely at } \mathbf 0.
\end{cases}
\end{equation}
As a consequence of \eqref{assumption one singular vertex}, the solution $\u$ of \eqref{weak continuous problem} is assumed to be analytic far from the singularity at~$\mathbf 0$. The general case of multiple corner singularities can be treated analogously.
The main result of the paper Theorem \ref{theorem exponential convergence}, namely the exponential convergence of the energy error in terms of the number of degrees of accuracy, remains valid also if $\u$ is singular at all the other vertices.

Henceforth, we also assume that the \Babuska and weighted Sobolev spaces introduced above are defined taking into account in their definition only the singularity at $\mathbf 0$.
More precisely, we define such spaces by modifying the weight function in \eqref{weight function vertices} into $\Phi_\beta(\xbold)=\min (1,\vert\xbold\vert)^\beta$, for some $\beta\in [0,1)$,
that is, the weight function associated only with the vertex $\mathbf 0$.

\section {Harmonic virtual element method with nonuniform degrees of accuracy} \label{section harmonic virtual element method with nonuniform degrees of accuracy}
In this section, we introduce a method for the approximation of problem \eqref{weak continuous problem} employing polygonal meshes. This method takes the name of harmonic virtual element method (henceforth harmonic VEM)
and is a modification of the standard virtual element method (henceforth VEM) tailored for the approximation of solutions to harmonic problem.

Let $\{\taun \}$ be a sequence of polygonal decompositions of $\O$. Let $\mathcal V_n$ ($\mathcal V_n^b$) and $\mathcal E_n$ ($\mathcal E_n^b$) be the set of (boundary) vertices and edges of decomposition $\taun$, respectively.
We assume that $\taun$ is a conforming decomposition for all $n \in \mathbb N_0$, that is to say that each boundary edge is an edge of only one element of $\taun$, whereas 
each internal edge is an edge of precisely two elements of $\taun$.
Given $\E \in \taun$, we denote by $\mathcal V^\E$ and $\mathcal E^\E$ the set of vertices and edges of the polygon $\E$.
Moreover, we set $h_\E := \text{diam}(\E)$ the diameter of polygon $\E$, for all $\E \in \taun$, and $h_\e := |\e|$ the length of edge $\e$, for all $\e \in \mathcal E^\E$.
Note that hanging nodes, i.e. multiple edges on a straight line, are allowed.

We require the following two assumptions on the polygonal decomposition $\taun$.
\begin{itemize}
\item [(\textbf{D1})] Every $\E \in \taun$ is star-shaped with respect to a ball of radius greater than or equal to $\rho_0 h_\E$, $\rho_0$ being a universal positive constant belonging to $(0, \frac{1}{2})$.
Since there are many possible balls satisfying the star-shapedeness condition we fix for each $\E \in \taun$ one ball $B=B(\E)$.
Furthermore, for all $\E\in \taun$ abutting $\mathbf 0$, the subtriangulation $\tautilden= \tautilden(\E)$ obtained by joining the vertices of $\E$ to $\mathbf 0$ is made of triangles that are star-shaped with respect to a ball of radius greater than or equal to $\rho_0 h_\T$,
$h_\T$ being diam($\T$) for all $\T \in \tautilden$. For all $\T \in \tautilden(\E)$, it holds $h_\E \approx h_\T$.

\item[(\textbf{D2})] For all edges $\e\in \mathcal E^\E$, $\E \in \taun$, it holds $h_\e \ge \rho_0 h_\E$, $\rho_0$ being the same constant in the assumption (\textbf{D1}).
Besides, the number of edges in \E, is uniformly bounded independently on the geometry of the domain.
\end{itemize}



We define the \emph{local} harmonic virtual element spaces. Given $\p \in \mathbb N$ and given the following space, defined on the boundary of a polygon $\E \in \taun$ as
\begin{equation} \label{boundary space}
\mathbb B(\partial \E) := \left\{ \vnn \in \mathcal C^0(\partial \E) \mid \vnn{}_ {|_\e} \in \mathbb P_{\pe} (\e),\, \forall \e \in \mathcal E^\E  \right\},
\end{equation}
we set
\begin{equation} \label{local harmonic space}
\VE := \left\{  \vnn \in H^1(\E) \mid \Delta \vnn =0,\, \vnn{}_{|_{\partial \E}}\in \mathbb B(\partial \E)  \right\}.
\end{equation}
The functions in $\VE$ are then the solutions to local Laplace problems with piecewise polynomial Dirichlet data; therefore, they are not known explicitly in closed form.

Let us consider the following set of linear functionals on $\VE$. Given $\vvv \in \VE$:
\begin{itemize}
\item the values of $\vvv$ at the vertices of $\E$;
\item the values of $\vvv$ at the $\pe - 1$ internal Gau\ss-Lobatto nodes of $\e$, for all $\e$ edges of $\E$.
\end{itemize}
This is a set of degrees of freedom, since  (i) the dimension of $\VE$ is equal to the number of functionals defined above and (ii) such functionals are uninsolvent,
owing to the fact that weak harmonic functions that vanish on $\partial \E$, vanish also in the interior of $\E$.
Thus, the dimension of space $\VE$ is finite and is equal to $\sum_{\e \in \mathcal E^\E} \pe = \pe \cdot \#(\text{edges of }\E)$.

We note that the definition of the edge degrees of freedom as the values of Gau\ss-Lobatto nodes is not the only possible;
for instance, modal degrees of freedom of integrated Legendre polynomials is suitable as well.

By $\dof_i$ we denote the $i$-th degree of freedom of $\VE$, whereas by $\{\varphi_i^\E\}_{i=1}^{\dim(\VE)}$ we denote the canonical basis of $\VE$, i.e. the set of basis functions in $\VE$ given by
\begin{equation} \label{canonical basis}
\dof_i(\varphi_j) = \delta_{i,j}, \quad i,j=1,\dots, \dim(\VE),
\end{equation}
where $\delta_{i,j}$ is the Kronecker delta.
We define the global harmonic virtual element space
\begin{equation} \label{general harmonic space}
\Vn := \left\{ \vn \in \mathcal C^0(\overline \O) \mid \vn{}_{|_\E} \in \VE,\, \forall \E \in \taun \right\},
\end{equation}
its subspace having vanishing boundary trace
\begin{equation} \label{harmonic test space}
\Vnzero := \left\{ \vn \in \Vn \mid  \vn{}_{|_{\partial \O}} = 0 \right\}
\end{equation}
and its affine subspace containing interpolated essential boundary conditions
\begin{equation} \label{global harmonic space}
\Vng := \left\{ \vn \in \Vn \mid  \vn{}_{|_\e} = \g_{GL}^\e\; \forall \e \in \mathcal E_n^b   \right\}.
\end{equation} 
Here $\g_{GL}^\e$ is the Gau\ss-Lobatto interpolant of degree $\pe$ of $\g$ on the edge $\e$ and where we recall $\mathcal E_n^b$ is the set of boundary edges of $\taun$.
We remark that $\g_{GL}^\e$ is well defined, since $\g \in \mathcal B_{\beta}^{\frac{3}{2}}(\Omega)$, which implies $\g \in \mathcal C^0(\overline \Omega)$, see \cite[Proposition 4.3]{SchwabpandhpFEM}.

The global degrees of freedom in the spaces \eqref{general harmonic space}, \eqref{harmonic test space} and \eqref{global harmonic space} are obtained by a standard continuous matching between the degrees of freedom of local spaces
and, in the latter case, by imposing proper polynomial Dirichlet boundary conditions.

The space $\Vnzero$ \eqref{harmonic test space} and the affine space $\Vng$ \eqref{global harmonic space} consist then of piecewise harmonic functions on each element, piecewise continuous polynomials on the skeleton
and piecewise Gau\ss-Lobatto interpolant of the Dirichlet datum $\g$ on the boundary. 
The name component ``virtual'' emphasizes that such functions are not known explicitly at the interior of each $\E \in \taun$,
since they are weak solutions to local Laplace problems with polynomial Dirichlet boundary conditions.
On the other hand, the name component ``harmonic'' emphasizes that functions in $\Vn$ are piecewise harmonic.

We point out that the choice of Gau\ss-Lobatto interpolation of the Dirichlet datum \eqref{global harmonic space} will play a role in the $\p$ approximation estimates.
However, other choices in order to have a proper $\p$ approximation of the boundary datum could be performed; for instance, one could use integrated Legendre polynomials interpolation of the Dirichlet datum as well.
We stick here to the choice of Gau\ss-Lobatto interpolation for the sake of clarity.

Having defined the approximation spaces, we introduce the harmonic VEM associated with \eqref{weak continuous problem}:
%
\begin{equation} \label{discrete problem}
\begin{cases}
\text{find } \un \in \Vng \text{ such that}\\
\an (\un, \vn) = 0 \quad \forall \vn \in \Vnzero,
\end{cases}
\end{equation}
where $\an(\cdot,\cdot)$ is an approximate symmetric bilinear form defined on the unrestricted space $\Vn \times \Vn$,  see \eqref{general harmonic space}.
We require that the bilinear form $\an(\cdot,\cdot)$ is \emph{explicitly computable by means of} the degrees of freedom of the space and it must mimic the properties of its continuous counterpart $a(\cdot,\cdot)$;
in particular, appropriate continuity and coercivity properties on $\an$ are required. We argue and derive a suitable representation of $\an(\cdot,\cdot)$ step-by-step.

First of all, we recall the representation
\[
a(u,v) = \sum_{\E \in \taun} \aE (u_{|_\E}, v_{|_\E}), \qquad \aE (u_{|_\E}, v_{|_\E}) := \int_\E \nabla u \cdot \nabla v \, dx.
\]
Thus it is natural to seek for $\an(\cdot,\cdot)$ as a sum of its local contributions
\begin{equation} \label{discrete bilinear form splitting}
\an(\un,\vn) = \sum_{\E \in \taun} \anE (\un{}_{|_\E}, \vn{}_{|_\E}) \quad \forall \un,\vn \in \Vn
\end{equation}
Here, the $\anE(\cdot,\cdot)$ are local discrete bilinear forms defined on $\VE \times \VE$.

Next, we impose the validity of the two following assumptions on $\anE(\un,\vn)$:
\begin{itemize}
\item [(\textbf{A1})] \textbf{local harmonic polynomial consistency}: for all $\E \in \taun$, it must hold
\begin{equation} \label{harmonic conistency}
\aE(\q,\vnn) = \anE(\q,\vnn) \qquad \forall \q \in \Har_{\pE} (\E),\; \forall \vnn \in \VE,
\end{equation}
where we recall that $\Har_{\pE}(\E)$ is the space of harmonic polynomials of degree $\pE$ over $\E$;
\item [(\textbf{A2})] \textbf{local stability}: for all $\E \in \taun$, it must hold
\begin{equation} \label{stability}
\alpha_*(\p) \vert \vnn \vert^2_{1,\E} \le \anE (\vnn, \vnn) \le \alpha^*(\p) \vert \vnn \vert^2_{1,\E} \quad \forall \vnn \in \VE,
\end{equation}
where $0 <\alpha_*(\p) \le \alpha^*(\p) < +\infty$ are two constants which may depend on the local space $\VE$. In particular, $\alpha_*$ and $\alpha^*$ must be independent of $\hE$.
\end{itemize}

The assumption (\textbf{A2}) is required to guarantee that the discrete bilinear form scales like its continuous counterpart.
In particular, it implies the coercivity and the continuity of the discrete bilinear form $\an$.
This, along with Lax Milgram lemma, implies the well-posedness of the problem \eqref{discrete problem}.

On the other hand, the assumption (\textbf{A1}) implies that the problem \eqref{discrete problem} passes the patch test, meaning that, if the solution to the continuous problem \eqref{weak continuous problem}
is a piecewise discontinuous harmonic polynomial, then the method described in \eqref{discrete problem} returns exactly, up to machine precision, the exact solution.
For this reason, $\p$ can be regarded as the degree of accuracy of the method.

We now investigate the behaviour of the error in the energy norm.
The following variation of the quasioptimality result for the discrete solution is an adaptation of \cite[Lemma 1]{hpVEMcorner}.
We define
\begin{equation} \label{alpha Strang}
\alpha(\p) := \frac{1+\alpha^*(\p)}{\alpha_*(\p)},
\end{equation}
where $\alpha_*(\p)$ and $\alpha^*(\p)$ are introduced in \eqref{stability}, and the $H^1$-broken Sobolev seminorm associated with the polygonal decomposition $\taun$
\begin{equation} \label{broken Sobolev}
\vert \vvv \vert^2_{1,\taun} := \sum_{\E \in \taun} \vert \vvv \vert^2_{1,\E}\quad \forall \vvv \in L^2(\Omega) \;\; \text{such that }\vvv{}_{|_\E} \in H^1(\E)\, \forall \,\E \in \taun.
\end{equation}
\begin{lem} \label{lemma Strang harmonic}
We assume that the assumptions (\textbf{A1}) and (\textbf{A2}) are satisfied. Let $\u$ and $\un$ be the solutions to problems \eqref{weak continuous problem} and \eqref{discrete problem}, respectively. Then, the following holds true:
\begin{equation} \label{two terms splitting}
\vert \u -\un \vert _{1,\O} \le \alpha(\p) \Big\{ \vert \u - \upi \vert_{1,\taun} + \vert \u - \uI \vert_{1,\Omega} \Big\}\qquad \forall \upi \in S^{\p,\Delta}(\O,\taun),\quad \forall \uI \in \Vng,
\end{equation}
where $\alpha(\p)$ and $\Vng$ are defined in \eqref{alpha Strang} and \eqref{global harmonic space}, respectively,
and where $S^{\p,\Delta}(\O,\taun)$ is the space of (globally discontinuous) piecewise harmonic polynomials of degree $\pE$ on each $\E\in \taun$.
\end{lem}
\begin{proof}
A triangle inequality yields
\begin{equation} \label{1st step abstract}
\vert u - \un \vert_{1,\Omega} \le \vert u - \uI \vert_{1,\Omega} + \vert \uI - \un \vert_{1,\Omega} \quad \forall \upi \in S^{\p,\Delta}(\O,\taun),\,\forall \uI \in \Vng.
\end{equation}
Owing to the assumptions (\textbf{A1}) and (\textbf{A2}), and the problems \eqref{weak continuous problem} and \eqref{discrete problem}, one gets
\begin{equation} \label{2nd step abstract}
\begin{split}
&\vert \uI - \un \vert ^2_{1,\O} 	 = \sum_{\E \in \taun} \vert \uI - \un \vert^2_{1,\E} \le  \sum_{\E \in \taun} \alpha^{-1}_*(\p) \left\{ \anE(\uI, \uI-\un) - \anE(\un, \uI - \un) \right\}\\
& = \alpha^{-1}_*(\p) \sum_{\E \in \taun} \left\{ \anE(\uI - \upi, \uI - \un) + \anE(\upi, \uI - \un)  \right\}	\\
& = \alpha_*^{-1}(\p) \sum_{\E \in \taun} \left\{ \anE(\uI-\upi, \uI-\un)  + \aE(\upi - \u, \uI-\un)\right\}\\
& \le \alpha_*^{-1}(\p) \sum_{\E \in \taun} \left\{ (1+\alpha^*(\p)) \vert \u -\upi \vert_{1,\E} \vert \uI -\un \vert_{1,\E} + \alpha^*(\p) \vert \u - \uI\vert_{1,\E} \vert \uI - \un \vert_{1,\E}\right\}.\\
\end{split}
\end{equation}
The claim follows plugging \eqref{2nd step abstract} in \eqref{1st step abstract} and from simple algebra.
\end{proof}

Lemma \ref{lemma Strang harmonic} states that the energy error arising from the method can be bounded by a sum of local contributions of best  local error terms
with respect to the space of harmonic polynomials and to the space of functions in the harmonic virtual element space \eqref{local harmonic space}.
We note that such best errors are weighted by the factor $\alpha(\p)$ defined in \eqref{alpha Strang}.

We exhibit now an explicit choice for $\anE(\cdot,\cdot)$. To this end, we need to define a local energy projection from the local harmonic virtual element space $\VE$ defined \eqref{local harmonic space} into $\Har_{\pE}(\E)$,
which we recall is the space of harmonic polynomials of degree $\pE$ over $\E$.
We then introduce the projector $\PinablaE_{\pE}$ defined as
\begin{equation} \label{nabla projector}
\PinablaE_{\pE} : \VE \rightarrow \Har_{\pE}(\E) \quad \text{such that} \quad
\begin{cases}
\aE(\q, \vnn - \PinablaE_{\pE} \vnn) = 0,\\
\int_{\partial \E} (\vnn - \PinablaE_{\pE} \vnn) ds = 0\\
\end{cases}\quad \forall \q \in \Har_{\pE}(\E),\; \forall \vnn \in \VE.
\end{equation}
The second equation in \eqref{nabla projector} only fixes constants and can be substituted by other \emph{computable} choices, see \cite{hitchhikersguideVEM, equivalentprojectorsforVEM}.
Henceforth, when no confusion occurs, we will write $\Pinabla$ in lieu of $\PinablaE_{\pE}$.

We note that the projector $\Pinabla$ can be computed by means of the dofs of space $\VE$. In fact, it suffices to apply an integration by parts to get
\[
\int_\E \nabla \q \cdot \nabla \vnn 
= \int_{\partial \E} (\partial_\n \q) \,\vnn \qquad \forall\, \q \in \Har_{\pE}(\E),\quad \forall \vnn \in \VE,
\]
where $\mathbf n$ denotes the exterior normal versor on the boundary of $\E$, $\partial _\n \q$ denotes the associated normal derivative
 and where we used that $\q$ is harmonic, i.e. $\Delta \q = 0$. In order to conclude, it suffices to note that both $\vnn$ and $\partial _\n \q$ are explicity known on $\partial \E$.

Let now $\SE: \ker (\Pinabla) \times \ker (\Pinabla) \rightarrow \mathbb R$ be any \emph{computable} bilinear form satisfying the following stability assumption:
\begin{equation} \label{local stabilization}
c_*(\p) \vert \vnn \vert ^2_{1,\E} \le \SE(\vnn,\vnn) \le c^*(\p) \vert \vnn \vert^2_{1,\E} \quad \forall \vnn \in \ker (\Pinabla),
\end{equation}
where $0 < c_*(\p) \le c^*(\p) < +\infty$ are two constants which may depend on the local space $\ker (\Pinabla)$.
An explicit selection for $\SE$ and a derivation of explicit bounds on $c_*(\p)$ and $c^*(\p)$ in terms of $\pE$ and $h_\E$ are the topic of Section \ref{subsection stability}.

At this point, we are ready to define the local discrete bilinear form. We set
\begin{equation} \label{local choice discrete bilinear form}
\anE(\unn,\vnn) = \aE(\Pinabla \unn, \Pinabla \vnn)  + \SE((I-\Pinabla) \unn, (I-\Pinabla) \vnn) \qquad \forall \unn,\vnn \in \VE.
\end{equation}
We observe that the local stability property \eqref{local stabilization} implies the validity of the assumptions~(\textbf{A1}) and~(\textbf{A2}). In particular, the assumption~(\textbf{A2}) holds with
\begin{equation} \label{link between alpha and c}
\alpha_*(\p) = \min (1, c_*(\p)),\qquad \alpha^*(\p) = \max (1,c^*(\p)).
\end{equation}
In Sections \ref{subsection stability} and \ref{subsection a p independent stabilization}, we investigate the behaviour of $\alpha(\p)$ in terms of $\pE$
for particular choices of the stabilization $\SE$ satisfying \eqref{local stabilization}.

\begin{remark} \label{remark mixed boundary conditions}
So far, we have assumed that the Laplace problem \eqref{string continuous formulation} is endowed with Dirichlet boundary conditions.
In the case of the Laplace problem $\Delta \u = 0$ with  mixed boundary conditions
\[
\begin{cases}
\u=\g_1 & \text{on }\Gamma_1,\\
\partial_\n\u=\g_2 & \text{on }\Gamma_2,\\
\end{cases}
\]
over two parts of the boundary $\partial \Omega = \overline{\Gamma_1} \cap \overline{\Gamma_2}$ having nonzero measure, the right-hand side of the weak formulation \eqref{discrete problem} is augmented by the term $(\g_2,\vvv)_{0,\Gamma_2}$.
\end{remark}

\subsection{{A stabilization with the $L^2$-norm on the skeleton}}
\label{subsection stability}
In this section we introduce a computable local stabilizing bilinear form $\SE$ satisfying \eqref{local stabilization}
and obtain explicit bounds in terms of the local degree of accuracy $\pE$ for the corresponding stabilization constants $c_*(\p)$ and $c^*(\p)$.
Our first candidate is
\begin{equation} \label{our choice stabilization}
\SE(\unn,\vnn) = \frac{\pE}{h_\E} (\unn, \vnn) _{0,\partial \E} = \frac{\pE}{h_\E} \sum_{\e \in \mathcal E^\E} (\unn,\vnn)_{0,\e} \quad \forall \unn,\,\vnn \in \VE.
\end{equation}
Since functions in $\VE$, defined in \eqref{local harmonic space}, are piecewise polynomials on the boundary of the element,
then it is clear that the local stabilization introduced in \eqref{our choice stabilization} is explicitly computable.

%

For computational purposes, we substitute the edge integrals on the right-hand side of \eqref{our choice stabilization} with Gau\ss-Lobatto quadratures. 
This new choice is spectrally equivalent to the one in \eqref{our choice stabilization}. Indeed, recalling \cite[(2.14)]{bernardimaday1992polynomialinterpolationinsobolev}
and setting $\widehat I = [-1,1]$, $\{\widehat \eta _j^\p\}_{i=0}^\p$ and $\{\widehat \xi _j^\p\}_{i=0}^\p$ the Gau\ss-Lobatto weights and nodes on $\widehat I$, then there exists a positive universal constant $\c$ such that
\begin{equation}\label{masslump}
\c \sum_{j=0}^\p \widehat \q ^2 (\widehat \xi _j^\p) \widehat \eta ^\p_j \le \Vert \q \Vert^2_{0,\widehat I} \le  \sum_{j=0}^\p \widehat \q ^2 (\widehat \xi _j^\p) \widehat \eta ^\p_j,\quad \forall \widehat \q \in \mathbb P_{\p} (\widehat I).
\end{equation}
A scaling argument in addition to the assumption (\textbf{D2}) guarantees that the terms of the sum on the right-hand side of \eqref{our choice stabilization} can be replaced with Gau\ss-Lobatto quadrature formulas.
This last choice is, from the computational point of view, more convenient than 
\eqref{our choice stabilization}, since it results in diagonal matrix blocks.
Thus, we emphasize our choice of $\SE$ by writing explicitly its definition. To each $\e \in \mathcal E^\E$ we associate the set of Gau\ss-Lobatto weights and nodes $\{ \eta_j^{\pe,\e} \}_{j=0}^{\pe}$ and $\{ \xi_j^{\pe,\e} \}_{j=0}^{\pe}$, respectively.
The local stabilizing bilinear form associated with method \eqref{discrete problem} reads
\begin{equation} \label{elegant choice stabilization}
\SE(\unn,\vnn) = \frac{\pE}{h_\E} \sum_{\e \in \mathcal E^\E} \left( \sum_{j=0}^{\pe} \eta_{j}^{\pe,\e} \unn (\xi_j^{\pe, \e}) \vnn (\xi_j^{\pe, \e})    \right).
\end{equation}
Next, we discuss the issue of showing explicit stability bounds \eqref{local stabilization} in terms of the local degree of accuracy.

Let us denote by 
\begin{equation} \label{average}
\overline \vvv := \frac{1}{|\E|} \int_\E \vvv
\end{equation}
the mean value of $v$ over $\E \in \taun$. Then the Poincar\'e inequality, see e.g.  \cite{BrennerScott}, implies
\begin{equation}\label{Poincare}
\Vert \vvv - \overline v \Vert_{0,\E} \lesssim h_\E \vert v \vert_{1,\E}\qquad \forall v \in H^1(\E).
\end{equation}
Moreover, when $\vvv \in \ker(\Pinabla)$, the following improved estimate is valid.
\begin{lem} \label{lemma Aubin Nitsche}
Let $\E\in \taun$ and let $\Pinabla$ be defined in \eqref{nabla projector}.
For any $\vvv \in \ker(\Pinabla)$, the following holds true:
\begin{equation}\label{AuNi}
\Vert \vvv - \overline v \Vert_{0,\E} \lesssim 
\begin{cases}
\hE \left( \frac{\log(\p)}{\p} \right)^{\frac{\lambda_\E}{\pi}} \vert \vvv \vert_{1,\E} & \text{if } \E \text{ is convex},\\
\hE \left( \frac{\log(\p)}{\p} \right)^{\frac{\lambda_\E}{\omega_\E}-\varepsilon} \vert \vvv \vert_{1,\E} & \forall \varepsilon > 0 \text{ arbitrarily small, }\text{otherwise},\\
\end{cases}
\end{equation}
where $\lambda_\E$ and $\omega_\E$ denote the smallest exterior and the largest interior angle of $\E$, respectively.
\end{lem}
\begin{proof}
We prove the assertion only for $\E$ convex, i.e. $0<\omega_\E < \pi$, since the nonconvex case can be treated analogously.
Moreover, we assume without loss of generality that $\hE=1$. The general form of the assertion \eqref{AuNi} follows then by a scaling argument.

The proof is based on an Aubin-Nitsche-type argument. For a fixed $\vvv \in \ker(\Pinabla)$, we consider an auxiliary problem of finding $\eta$ such that
\begin{equation} \label{dual problem Neumann}
\begin{cases}
-\Delta \eta = \vvv - \overline \vvv& \text{in } \E ,\\
\partial_\n \eta = 0 & \text{on } \partial \E ,\\
\int_\E \eta = 0,\\
\end{cases}
\end{equation}
where we recall that $\overline \vvv$ is defined in \eqref{average}.

We observe that by construction the right-hand side in \eqref{dual problem Neumann} has vanishing mean and thus by the Lax-Milgram lemma the solution $\eta \in H^1(K)$ is well defined.
The additional regularity of $\eta$ depends on the size of interior angles of $\E$. In particular, if $\E$ is convex, there holds $\eta \in H^2(\E)$. More precisely,
\begin{equation} \label{a priori estimate for Neumann problem}
\Vert \eta \Vert_{2,\E} \lesssim \Vert v -\overline v \Vert_{0,\E},
\end{equation}
see e.g. \cite[Section 4.2]{SchwabpandhpFEM}.

In the following, we utilize the additive splitting $\eta=\eta_1+\eta_2$, where the summands satisfy
\[
\begin{cases}
-\Delta \eta_1 = \vvv - \overline \vvv & \text{in }\E,\\
\eta_1 = 0 & \text{on } \partial \E,
\end{cases}\quad \quad
\begin{cases}
-\Delta \eta_2 = 0 & \text{in }\E,\\
\eta_1 = \eta & \text{on } \partial \E.
\end{cases}
\]
Again, standard a priori regularity results entail for a convex $\E$
\begin{equation} \label{cattiveria}
\Vert \eta_1 \Vert_{2,\E} \lesssim \Vert \vvv - \overline \vvv \Vert_{0,\E}.
\end{equation}
Therefore, a combination of \eqref{a priori estimate for Neumann problem} and \eqref{cattiveria} with a triangle inequality, yields
\begin{equation} \label{bababa}
\Vert \eta_2 \Vert_{2,\E} \le \Vert \eta \Vert_{2,\E} + \Vert \eta_1 \Vert_{2,\E} \lesssim \Vert \vvv - \overline \vvv \Vert_{0,\E}.
\end{equation}
Besides, given any $w \in H^1(\E)$ which is also harmonic, one has
\begin{equation} \label{schmerz}
(\nabla \eta_1, \nabla w)_{0,\E} = (\eta_1, \partial_{\n} w)_{0,\partial \E} - (\eta_1,  \Delta w)_{0,\E} =0.
\end{equation}
Recalling that $\vvv \in \ker (\Pinabla)$ and applying sequentially \eqref{dual problem Neumann}, an integration by parts, \eqref{schmerz}, orthogonality of $\Pinabla$, the Cauchy-Schwarz inequality and \cite[Theorem 2]{babumelenk_harmonicpolynomials_approx}, we deduce
\begin{equation} \label{delirium}
\begin{split}
& \Vert \vvv - \overline \vvv \Vert^2_{0,\E} = (-\Delta \eta, \vvv)_{0,\E} = (\nabla \eta, \nabla(\vvv - \overline \vvv))_{0,\E} = (\nabla \eta_2, \nabla(\vvv - \overline \vvv))_{0,\E}   = (\nabla \eta_2, \nabla(\vvv - \Pinabla \vvv))_{0,\E} \\
& = (\nabla (\eta_2 - \Pinabla \eta_2), \nabla \vvv )_{0,\E} \le \vert \eta_2 - \Pinabla \eta_2 \vert_{1,\E} \vert \vvv \vert_{1,\E} \lesssim \left( \frac{\log(\p)}{\p} \right) ^{\frac{\lambda_\E}{\pi}} \Vert  \eta_2 \Vert_{2,\E} \vert \vvv \vert_{1,\E},
\end{split}
\end{equation}
where $\lambda_\E$ denotes the smallest exterior angle of $\E$.

Plugging \eqref{delirium} in \eqref{bababa}, we get the assertion.
\end{proof}

Now, we are ready to prove stability estimates for the spectrally equivalent $L^2$-norm stabilizations introduced in \eqref{our choice stabilization} and \eqref{elegant choice stabilization}.

\begin{lem} \label{lemma stability bounds}
The bilinear forms $\SE$ defined in \eqref{our choice stabilization} and \eqref{elegant choice stabilization} fulfill the two-sided estimate \eqref{local stabilization} with constants satisfying
\begin{equation} \label{explicit bounds stability}
c_*(\p) \gtrsim \pE^{-1},\quad \quad \quad \c^*(\p) \lesssim
\begin{cases}
\p \left( \frac{\log(\p)}{\p}  \right)^{\frac{\lambda_\E}{\pi}} & \text{if } \E \text{ is convex},\\
\p \left( \frac{\log(\p)}{\p}  \right)^{\frac{\lambda_\E}{\omega_\E} - \varepsilon} & \forall \varepsilon >0 \text{ arbitrarily small, }\text{otherwise,}\\
\end{cases}
\end{equation}
where $\lambda_\E$ and $\omega_\E$ denote the smallest exterior and the largest interior angles of $\E$, respectively.
\end{lem}
\begin{proof}
We prove the assertion only for $\E$ convex, i.e. $0<\omega_\E < \pi$, since the nonconvex case can be treated analogously.
Moreover, in view of \eqref{masslump}, it suffices to consider the bilinear form $\SE$ from \eqref{our choice stabilization}. We also assume $h_\E$=1 since the  assertion will follow by a scaling argument.

We start by proving the lower bound for $c_*(\p)$. Given $\vnn \in \ker (\Pinabla)$, we write
\begin{equation} \label{energy weak harmonic}
\vert \vnn \vert^2_{1,\E} = \int _\E \nabla \vnn \cdot \nabla \vnn = \int_{\partial \E} (\partial_\n \vnn) \, \vnn,
\end{equation}
where we used an integration by parts and the fact that $\vnn$ is harmonic in $\E$.
We apply now a Neumann trace inequality \cite[Theorem A33]{SchwabpandhpFEM} with $\Delta \vnn=0$ in $\E$, in order to show that
\begin{equation} \label{bounding the boundary}
\int_{\partial \E} (\partial _\n \vnn) \, \vnn \le \left\Vert \partial_\n \vnn \right\Vert_{-\frac{1}{2},\partial \E} \Vert \vnn \Vert_{\frac{1}{2}, \partial \E} \lesssim \vert \vnn \vert_{1,\E} \Vert \vnn \Vert_{\frac{1}{2},\partial \E}.
\end{equation}
Plugging \eqref{bounding the boundary} in \eqref{energy weak harmonic} and using the polynomial $\h\p$ inverse inequality on an interval \cite[Theorem 3.91]{SchwabpandhpFEM} and interpolation theory \cite{Triebel}, we obtain
\[
\vert \vnn \vert_{1,\E}^2 \lesssim \Vert \vnn \Vert_{\frac{1}{2},\partial \E}^2 \lesssim \pE^2 \Vert \vnn \Vert_{0, \partial \E}^2 = \pE \cdot \SE(\vnn,\vnn),
\]
which is the asserted bound on $c_*(\p)$.

Next, we investigate the behaviour of $c^*(\p)$. 
Given $\vvv \in \ker (\Pinabla)$ and $\overline \vvv$  defined as in \eqref{average}, one has
\begin{equation} \label{eq1}
\SE(\vnn, \vnn) = \pE \Vert \vnn \Vert_{0,\partial \E}^2 \lesssim
p (\Vert v - \overline v \Vert_{0,\partial \E}^2 + \vert \partial \E\vert \cdot \vert\overline v\vert^2).
\end{equation}
We observe that, by \eqref{nabla projector}, $v$ has zero boundary mean and therefore, by the Cauchy-Schwarz inequality,
\begin{equation} \label{eq2}
|\partial \E| \cdot |\overline v|^2 = \frac{1}{|\partial \E|} \cdot \left|\int_{\partial \E}(v-\overline v) \right|^2 \leq \|v - \overline v\|_{0,\partial \E}^2.
\end{equation}
Hence by \eqref{eq1}, \eqref{eq2}, the multiplicative trace inequality and \eqref{AuNi}, we deduce
\begin{equation}
\SE(\vnn, \vnn) \lesssim  \pE \vert \vvv - \overline \vvv \vert_{0,\partial \E}^2 \lesssim \p (\Vert \vvv - \overline \vvv \Vert_{0,\E} \vert \vvv \vert_{1,\E} + \Vert \vvv - \overline \vvv \Vert^2_{0,\E})
\lesssim \p \left( \frac{\log(\p)}{\p}   \right)^{\frac{\lambda_\E}{\pi}} \vert \vvv \vert^2_{1,\E},
\end{equation}
where $\lambda_\E$ denotes the smallest exterior angle of $\E$.
\end{proof}
Lemma \ref{lemma stability bounds} and \eqref{link between alpha and c} imply that $\alpha(\p)$ introduced in \eqref{alpha Strang} admits the upper bound
\begin{equation} \label{explicit bound alpha}
\alpha(\p) := \frac{1+\alpha^*(\p)}{\alpha_*(\p)} \lesssim
\begin{cases}
\p^2\left( \frac{\log(\p)}{\p}  \right)^{\frac{\lambda_\E}{\pi}} & \text{if all } \E\in \taun \text{ are convex},\\
\p^2\left( \frac{\log(\p)}{\p}  \right)^{\min_{\E \in \taun} \frac{\lambda_\E}{\omega_\E} - \varepsilon} & \forall \varepsilon >0 \text{ arbitrarily small, } \text{otherwise},\\
\end{cases}
\end{equation}
where $\lambda_\E$ and $\omega_\E$ denote the smallest exterior and largest interior angles of $\E$, for all $\E \in \taun$, respectively.

We emphasize that the corresponding stability constant obtained for the standard (i.e. nonharmonic) $\h\p$ virtual element method, see \cite[Theorem 2]{hpVEMcorner}, grows much faster in $\p$ than $\alpha(\p)$.
More precisely, it was proven that
\[
\alpha(\p) \lesssim
\begin{cases}
\p^5 & \text{if all } \E\in \taun \text{ are convex},\\
\p^{2\max_{\E\in \taun} \left( 1 - \frac{\pi}{\omega_\E} - \varepsilon     \right) + 5} & \forall \varepsilon >0 \text{ arbitrarily small, }\text{otherwise},
\end{cases}
\]
where, for all $\E \in \taun$, $\omega_\E$ denotes the largest interior angle of $\E$.

We conclude this section by noting that the stabilization introduced in \eqref{our choice stabilization} is basically, up to a $\p$ scaling, the weighted (with Gau\ss-Lobatto weights) boundary contribution
of the standard VEM stabilization introduced in \cite{VEMvolley, hitchhikersguideVEM}.

\subsection{An optimal stabilization with the $H^{1/2}$-norm on the skeleton}
\label{subsection a p independent stabilization}
In view of Theorem \ref{theorem exponential convergence}, which guarantees exponential convergence of the method in terms of the number of degrees of freedom, the mild behaviour of the stability constants $c_*(\p)$ and $c^*(\p)$
described in Lemma \ref{lemma stability bounds} in terms of $\p$ has no effect on the asymptotic convergence rate of the method this remains exponential.

However, it is worth mentioning that there exists an optimal stabilization bilinear form $\SE$ with uniformly bounded stability constants $c_*$ and $c^*$; such stabilization reads
\begin{equation} \label{p independent stab}
\SE(\unn,\vnn) = (\unn, \vnn)_{\frac{1}{2},\partial \E}\quad \forall \unn,\, \vnn \in \ker(\Pinabla),
\end{equation}
where $(\cdot, \cdot)_{\frac{1}{2},\partial \E}$ in the inner product on the Hilbert space $H^{\frac{1}{2}}(\partial \E)$.
\begin{lem} \label{lemma p independent stabilization}
Let $\SE$ be defined as in \eqref{p independent stab}. Then, for all $\vnn \in \ker (\Pinabla)$, $\Pinabla$ being defined in \eqref{nabla projector}, the following holds true:
\[
\SE(\vnn, \vnn) \approx \vert \vnn \vert^2_{1,\E}.
\]
\end{lem}
\begin{proof}
The statement follows from the proof of Lemma \ref{lemma stability bounds} and a scaling argument.
\end{proof}

It can be expected that the evaluation of \eqref{p independent stab} is more involved than the evaluation of the other variants of stabilization presented in Section \ref{subsection stability}, namely those in
\eqref{our choice stabilization} and \eqref{elegant choice stabilization}. In the following, we briefly discuss evaluation of the local stabilization \eqref{p independent stab}.

We firstly recall the definition of the Aronszajn-Slobodeckij $H^{\frac{1}{2}}$ inner product over $\partial \E$
\begin{equation} \label{H12 inner product}
\begin{split}
(\unn, \vnn)_{\frac{1}{2},\partial \E} 	& = (\unn, \vnn)_{0,\partial \E} + \int_{\partial \E} \int_{\partial \E} \frac{(\unn(\xi)- \unn(\eta)) (\vnn(\xi) - \vnn (\eta))}{\vert \xi -\eta\vert^2} \, d\xi\, d\eta\\
						& => (\unn, \vnn)_{0,\partial \E} + \sum_{\e_i = 1}^ {\NeE}\sum_{\e_j=1}^{\NeE} I_{ij}, \qquad I_{ij} := \int_{\e_i} \int_{\e_j} \frac{(\unn(\xi)- \unn(\eta)) (\vnn(\xi) - \vnn (\eta))}{\vert \xi -\eta\vert^2} \, d\xi\, d\eta,\\
\end{split}
\end{equation}
where $\NeE$ denotes the number of edges of $\E$ and $\{\e_i\}_{i=1}^{\NeE}$ denotes its set of edges.
We observe that, owing to the fact that the stabilization is defined on $\ker(\Pinabla)$, it is possible to drop in \eqref{H12 inner product} the contribution of the $L^2$ inner product.

We discuss now the evaluation of the double integral $I_{ij}$ in \eqref{H12 inner product}. We distinguish three different variants of the mutual locations of two edges $s_i$ and $s_j$.
\begin{enumerate}
\item $s_i$ and $s_j$ are identical ($s_i \equiv s_j$).
In this case, the integrand in \eqref{H12 inner product} has a removable singularity and is, in fact, a polynomial of degree $2p-2$. Such an integral is computed \emph{exactly} by means of a Gau\ss-Lobatto quadrature formula with $\p+1$ points.
\item $s_i$ and $s_j$ are distant ($\overline s_i \cap \overline s_j = \emptyset$). In this case, the integrand in \eqref{H12 inner product} is an analytic function and can be efficiently approximated e.g. by a Gau\ss-Lobatto quadrature rule,
see e.g. \cite[Theorem 5.4]{ChernovSchwab_exponentialGJquadrature}.

\item $s_i$ and $s_j$ share a vertex $\vec{v}$ and make an interior angle $0< \varphi < 2\pi$. Then, $s_i$ and $s_j$ admit local parametrizations
\begin{equation}
s_i = \{\xi = \vec{v} + \vec{a}s \mid 0 < s < 1\}, \qquad
s_j = \{\eta = \vec{v} + \vec{b}t \mid 0 < t < 1\},
\end{equation}
for some $\vec{a}$ and $\vec{b} \in \mathbb R^2$.
Since the functions $u,v \in \VE$ are polynomials of degree $p$ along $s_i$ and $s_j$ and are continuous in $\vec{v}$ there holds
\begin{equation}
u(\xi) - u(\eta) = s\,f(s) - t\,g(t), \qquad
v(\xi) - v(\eta) = s\,q(s) - t\,r(t),
\end{equation}
where $f, g, q$ and $r$ are polynomials of degree $p-1$ and one has, using a change of coordinate,
\begin{equation}
I_{ij} = |\vec{a}|\cdot|\vec{b}| \int_0^1 \int_0^1 F(s,t) \, ds dt, \quad \text{where} \quad F(s,t) = \frac{\big(s\,f(s) - t\,g(t)\big)\big(s\,q(s) - t\,r(t)\big)}{|\vec{a} s - \vec{b} t|^2}.
\end{equation}

The integrand $F(s,t)$ is not smooth in $(0,1)^2$ (its derivatives blow up near the origin) and is not even defined in the origin, but it becomes regular after a coordinate transformation \cite{duffy}.
Having split the integral over the square $(0,1)^2$ into a sum of integrals over the two triangles obtained by bisecting such square with the segment of endpoints $(0,0)$ and $(1,1)$, simple algebra yields
\begin{equation} \label{regularized-common-vtx}
\begin{split}
I_{ij} & = |\vec{a}|\cdot|\vec{b}|  \int_0^1 \int_0^t \big( F(s,t) + F(t,s)\big) \, ds dt \\
& = |\vec{a}|\cdot|\vec{b}|  \int_0^1 \int_0^1 t \cdot \big( F(tz,t) + F(t,tz)\big)  \, dz dt,
\end{split}
\end{equation}
after the transformation $s = tz$ in the inner integral. The integrand admits the representation
\begin{equation}
F(tz,t) = \frac{\big(z\,f(tz) - g(t)\big)\big(z\,q(tz) - r(t)\big)}{|\vec{a} z - \vec{b}|^2},
\end{equation}
which is a rational function with a uniformly positive denominator
\begin{equation}
|\vec{a} z - \vec{b}|^2 \geq 
\left\{
\begin{array}{ll}
|\vec{b}|^2 \sin^2 \varphi,  & \text{ for } \cos \varphi > 0\\
|\vec{b}|^2,               & \text{ for } \cos \varphi \leq 0
\end{array}
\right\}> 0. 
\end{equation}
Hence, the integrand \eqref{regularized-common-vtx} is an analytic function and can be efficiently approximated by Gau\ss-Lobatto quadrature, see e.g. \cite{bernardimaday1992polynomialinterpolationinsobolev}.
\end{enumerate}
\begin{remark}\label{remark stabilization valid also in VEM}
In \cite{beiraolovadinarusso_stabilityVEM}, in the context of the approximation of a 2D Poisson problem,
the possibility of using a stabilization involving only the boundary degrees of freedom was proven. More precisely, a stabilization equal to the boundary $H^1$ norm was employed;
such norm can be related to the one introduced in \eqref{p independent stab} via $\h\p$ polynomial inverse estimates in one dimension.
However, the analysis of \cite{beiraolovadinarusso_stabilityVEM} is not proven for the $p$ version of the method and therefore it is not clear whether the boundary stabilization therein proposed can be employed also for the $p$ analysis.
\end{remark}

\section {Exponential convergence with geometric graded polygonal meshes} \label{section exponential convergence}
In this section, we prove that, employing geometric refined towards $\mathbf 0$ meshes and choosing appropriately a distribution of local degrees of accuracy,
lead to exponential convergence of the energy error in terms of the dimension of the space, that is, in terms of the number of degrees of freedom.

We split the analysis as follows.
In Section \ref{subsection geometric meshes}, we introduce the concept of sequences of polygonal meshes that are geometrically graded towards $\mathbf 0$ (we recall that we are assuming that $\mathbf 0$ is the unique ``singular vertex'' of $\O$, see \eqref{assumption one singular vertex}).
In Section \ref{subsection approximation by harmonic polynomials}, we discuss the approximation results by harmonic polynomials,
whereas in Section \ref{subsection approximation by functions in the Harmonic Element Space} we discuss the approximation results by functions in the harmonic virtual element space.
Finally, in Section \ref{subsection exponential convergence}, we prove, under a proper choice of the vector of the degrees of accuracy, exponential convergence of the energy error in terms of the number of the degrees of freedom.

\subsection{Geometric meshes} \label{subsection geometric meshes}
We describe sequences of geometrically graded meshes that we will employ for proving Theorem~\ref{theorem exponential convergence}.
Let~$\sigma \in (0,1)$ be a given parameter. The sequence~$\{\taun\}$ is such that~$\taun$ consists then of~$n+1$ ``layers'' for every $n \in \mathbb N_0$, where the ``layers'' are defined as follows.

We set the $0$-th layer $L_{n,0}=L_0$ as the set of all polygons $\E \in \taun$ abutting $\mathbf 0$, which we recall is the unique ``singular corner'' of $\O$ by the assumption \eqref{assumption one singular vertex}.
The other layers are defined by induction as
\begin{equation} \label{jth layer}
L_{n,j} = L_j := \{ \E_1 \in \taun \mid \overline \E_1 \cap \overline \E_2 \neq \emptyset \text{ for some }  \E_2 \in L_{j-1} \text{ and } \E_1 \nsubseteq \cup_{i=0}^{j-1} L_i \} \quad \forall j=1,\dots, n.
\end{equation}
Next, we describe a procedure for building geometric polygonal graded meshes.
Let $\mathcal T_0 = \{\O\}$. The decomposition $\mathcal T_{n+1}$ is obtained by refining decomposition $\taun$ \emph{only} at the elements in the finest layer~$L_0$.
In order to have a proper geometric graded sequence of nested meshes, we demand for the following assumption.
\begin{itemize}
\item[(\textbf{D3})]
\begin{equation} \label{assumption geometric graded}
h_\E \approx
\begin{cases}
\sigma^n & \text{if } \E \in L_0,\\
\frac{1-\sigma}{\sigma} \dist (\E,\mathbf 0) & \text{if } \E \in L_j,\quad j=1,\dots,n .
\end{cases}
\end{equation}
\end{itemize}
A consequence of the assumption (\textbf{D3}) is that $h_\E \approx \sigma^{n-j}$, $j$ being the layer to which $\E$ belongs.
This, in addition to~\eqref{assumption geometric graded} guarantees that the distance between $\E \in L_j$, $j=1,\dots,n$ and $\mathbf 0$ is proportional to $\sigma ^{n-j}$.
Moreover, following \cite[(5.6)]{hmps_harmonicpolynomialsapproximationandTrefftzhpdgFEM}, it can be shown that the number of elements in each layer is uniformly bounded with respect to all the geometric parameters discussed so far.

The sequence of nested meshes that we build is then characterized by very small elements near the singularity, whereas the size of the elements increases proportionally with the distance between the elements theirselves and $\mathbf 0$.
\exmp \label{example meshes}
In Figure \ref{figure possible geometric decompositions}, we depict three polygonal meshes satisfying the assumption (\textbf{D3}).
We observe that the mesh in Figure \ref{figure possible geometric decompositions} (right) does not fulfill the star-shapedness assumption (\textbf{D1}).
We depict with different colours polygons belonging to different layers.

\begin{figure}  [h]
\begin{center}
\begin{minipage}{0.30\textwidth}
\begin{tikzpicture}[scale=0.4]
\fill[red,opacity=0.3] (0,-1) -- (1,-1) -- (1,1) -- (-1,1) -- (-1,0) -- (0,0);
\fill[blue,opacity=0.3] (0,-1) -- (0,-2) -- (2,-2) -- (2,2) -- (-2,2) -- (-2,0) -- (-1,0) -- (-1,1) -- (1,1) -- (1,-1);
\fill[yellow,opacity=0.3] (0,-4) -- (0,-2) -- (2,-2) -- (2,2) -- (-2,2) -- (-2,0) -- (-4,0) -- (-4,4) -- (4,4) -- (4,-4);
\draw[black, very thick, -] (0,0) -- (0,-4) -- (4,-4) -- (4,4) -- (-4,4) -- (-4,0) -- (0,0);
\draw[black, very thick, -] (0,0) -- (4,0); \draw[black, very thick, -] (0,0) -- (0,4);
\draw[black, very thick, -] (-4,2) -- (4,2); \draw[black, very thick, -] (2,-4) -- (2,4); \draw[black, very thick, -] (-2,0) -- (-2,4); \draw[black, very thick, -] (0,-2) -- (4,-2);
\draw[black, very thick, -] (-2,1) -- (2,1); \draw[black, very thick, -] (1,2) -- (1,-2);
\draw[black, very thick, -] (-1,0) -- (-1,2); \draw[black, very thick, -] (0,-1) -- (2,-1);
\end{tikzpicture}
\end{minipage}
\begin{minipage}{0.30\textwidth}
\begin{tikzpicture}[scale=0.4]
\fill[red,opacity=0.3] (0,-1) -- (1,-1) -- (1,1) -- (-1,1) -- (-1,0) -- (0,0);
\fill[blue,opacity=0.3] (0,-1) -- (0,-2) -- (2,-2) -- (2,2) -- (-2,2) -- (-2,0) -- (-1,0) -- (-1,1) -- (1,1) -- (1,-1);
\fill[yellow,opacity=0.3] (0,-4) -- (0,-2) -- (2,-2) -- (2,2) -- (-2,2) -- (-2,0) -- (-4,0) -- (-4,4) -- (4,4) -- (4,-4);
\draw[black, very thick, -] (0,0) -- (0,-4) -- (4,-4) -- (4,4) -- (-4,4) -- (-4,0) -- (0,0);
\draw[black, very thick, -] (0,-2) -- (2,-2) -- (2,2) -- (-2,2) -- (-2,0);
\draw[black, very thick, -] (0,-1) -- (1,-1) -- (1,1) -- (-1,1) -- (-1,0);
\draw[black, very thick, -] (0,0) -- (4,4);
\end{tikzpicture}
\end{minipage}
\begin{minipage}{0.33\textwidth}
\begin{tikzpicture}[scale=0.4]
\fill[red,opacity=0.3] (0,-1) -- (1,-1) -- (1,1) -- (-1,1) -- (-1,0) -- (0,0);
\fill[blue,opacity=0.3] (0,-1) -- (0,-2) -- (2,-2) -- (2,2) -- (-2,2) -- (-2,0) -- (-1,0) -- (-1,1) -- (1,1) -- (1,-1);
\fill[yellow,opacity=0.3] (0,-4) -- (0,-2) -- (2,-2) -- (2,2) -- (-2,2) -- (-2,0) -- (-4,0) -- (-4,4) -- (4,4) -- (4,-4);
\draw[black, very thick, -] (0,0) -- (0,-4) -- (4,-4) -- (4,4) -- (-4,4) -- (-4,0) -- (0,0);
\draw[black, very thick, -] (0,-2) -- (2,-2) -- (2,2) -- (-2,2) -- (-2,0);
\draw[black, very thick, -] (0,-1) -- (1,-1) -- (1,1) -- (-1,1) -- (-1,0);
\end{tikzpicture}
\end{minipage}
\end{center}
\caption{Decomposition $\taun$, $n=3$, made of: squares (left), nonconvex hexagons and quadrilaterals (center), nonstar-shaped/nonconvex decagons and nonstar-shaped/nonconvex hexagons (right).
The $0$-th, $1$-st and $2$-nd layers are coloured in light red, blue and yellow, respectively.}
\label{figure possible geometric decompositions}
\end{figure}
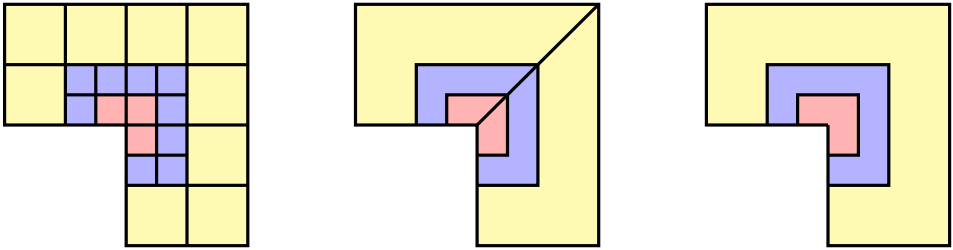
\subsection{Approximation by harmonic polynomials} \label{subsection approximation by harmonic polynomials}
Here, we discuss approximation estimates by means of harmonic polynomials. Such results will be used for the approximation of the first term in \eqref{two terms splitting},
that is the best approximation in the $H^1$ seminorm of the solution to \eqref{weak continuous problem} by harmonic polynomials.

We will firstly deal with the approximation by harmonic polynomials on the polygons that are far from the singularity, see Lemma \ref{lemma approximation harmonic polynomials far from singularity}.
Secondly, we will discuss the approximation estimates by harmonic polynomials on the polygons abutting the singularity, see Lemma \ref{lemma harmonic polynomials near singularity}.

Before that, we recall a (technical) auxiliary result, involving approximation on a polygon $\E$ with $h_\E = 1$ by means of harmonic polynomials.
The proof of this theorem can be found in \cite[Theorem 4.10]{hmps_harmonicpolynomialsapproximationandTrefftzhpdgFEM} and relies on the results in the pioneering works of \cite{melenk1999operator,melenk_phdthesis}.

\begin{thm} \label{theorem HMPS21}
Let $\Ehat$ be a polygon with $h_{\Ehat} = 1$. In particular, meas$(\Ehat)$$ < 1$.
We assume that the following parameters are given:
\begin{equation} \label{parameters HMPS21}
\begin{split}
& \delta \in \left(0,\frac{1}{2}\right]; \qquad
\xi = \begin{cases} 
1 & \text{if } \Ehat \text{ is convex},\\
\frac{2}{\pi} \arcsin{\left( \frac{\rho_0}{1-\rho_0} \right)} & \text{otherwise};
\end{cases} \qquad
c_{\Ehat} = \frac{27}{\xi};\\
& \overline r < \min \left( \frac{1}{3} \left( \frac{\delta}{c_{\Ehat}}\right) ^{\frac{1}{\xi}} , \frac{\rho_0}{4} \right) ; \qquad c_I = \frac{\rho_0}{4};\qquad c_{\text{approx}} \le \frac{7}{\rho_0^2};\qquad \gamma \le  \frac{72}{\rho_0^4},\\
\end{split}
\end{equation}
where we recall that $\rho_0$ is the radius of the ball with respect to which $\Ehat$ is star shaped, see the assumption~(\textbf{D1}).
Let also:
\begin{equation} \label{extension polygon}
\Ehat_{\delta} := \left\{ \widehat {\xbold} \in \mathbb R^2 \mid \dist (\Ehat, \widehat {\xbold}) < \delta \right\}.
\end{equation}
Then, there exists a sequence $\{\qhat_{\p}\}_{\p=1}^{\infty}$, $\qhat_\p  \in \mathbb H_\p (\Ehat)$ for all $\p \in \mathbb N$, of harmonic polynomials such that, for any $\uhat \in W^{1,\infty}(\Ehat _\delta)$,
\begin{equation} \label{estimate with harmonic polynomials diam 1}
\vert \uhat - \qhat_\p \vert_{1,\Ehat} \le \sqrt 2 c_{\text{appr}} \frac{2}{c_I \overline r ^2} \overline r^{-\gamma} (1+ \overline r)^{-\pE} \Vert \uhat \Vert_{W^{1,\infty}(\Ehat_{\delta})}.
\end{equation}
\end{thm}
\medskip\medskip\medskip

We do not discuss the proof of Thorem \ref{theorem HMPS21},	but we point out that in order to have this result we are using the fact that $\rho_0$ introduced in the assumption~(\textbf{D1}) is such that $\rho_0\in (0,\frac{1}{2})$,
since \cite[Theorem 4.10]{hmps_harmonicpolynomialsapproximationandTrefftzhpdgFEM} holds true under this hypothesis.

As a consequence of Theorem \ref{theorem HMPS21}, for all the regular (in the sense of the assumptions~(\textbf{D1}) and~(\textbf{D2})) polygons $\Ehat$ with diameter $1$ it holds that there exists an harmonic polynomial $\q_{\pE}$ of degree $\pE$ such that
\begin{equation} \label{estimate with harmonic polynomials exponential convergence}
\vert \uhat - \qhat_{\pE} \vert_{1,\Ehat} \le \c \exp{(-b \pE)} \Vert \uhat \Vert_{W^{1,\infty}(\Ehat_{\delta})},
\end{equation}
where $c$ and $b$ are two positive constants depending uniquely on $\rho_0$ introduced in the assumption~(\textbf{D1}) and the ``enlargement factor'' $\delta$ introduced in~\eqref{parameters HMPS21}.
Since both~$\rho_0$ and~$\delta$ are for the time being fixed, then~$c$ and~$b$ are two positive universal constants.

We assume now that the polygon~$\E$ belongs to~$L_j$, $j=1,\dots,n$ and consequently has the diameter unequal to~$1$.
Then, a scaling argument immediately yields
\begin{equation} \label{same inequality as above for variable diameter}
\vert \u - \q_{\pE} \vert_{1,\E} \approx \vert \uhat - \qhat_{\pEhat} \vert_{1,\Ehat} \lesssim \exp{(-b \pEhat)} \Vert \uhat \Vert_{W^{1,\infty}(\Ehat_\delta)} \lesssim \h_{\E_\varepsilon} \exp{(-b \pE)} \Vert \u \Vert_{W^{1,\infty}(\E_\varepsilon)},
\end{equation}
where $\Ehat$, the polygon obtained by scaling $\E$, is such that $h_{\Ehat} = 1$, where $\{\qhat _{\pEhat}\}_{\pEhat = 1} ^{\infty}$ is the sequence validating \eqref{estimate with harmonic polynomials exponential convergence},
where $\E_\varepsilon$ is defined as in \eqref{extension polygon} 
and where the ``enlargement'' factor $\varepsilon$ must be chosen in such a way that when we scale $\E$ to $\Ehat$, then $\E_{\varepsilon}$ is mapped in $\Ehat_{\delta}$, $\delta$ being \emph{exactly} the parameter fixed in \eqref{parameters HMPS21}.

We note that sequence $\{\q_{\pE}\}_{\pE=1}^{\infty}$, which is the pull-back of $\{\qhat_{\pE}\}_{\pE=1}^{\infty}$, consists of harmonic polynomials since it is the composition of a sequence of harmonic polynomials with a dilatation.

What we have to check is that the size of $\E_{\varepsilon}$ is not too large. In particular, we want that $\E_\varepsilon$ is kept separated from the singularity at $\mathbf 0$, for all $L_j$, $j=1,\dots, n$.

Let $\u$ be the solution to problem \eqref{weak continuous problem}.
Henceforth, we assume that $\dist(\E,\mathbf 0) < 1$ (which is always valid if one takes $\Omega$, the domain of problem \eqref{string continuous formulation}, small enough).
From Section \ref{section the model problem}, we know that $\u$, the solution to problem \eqref{string continuous formulation}, is analytic on the set $\mathcal N(\u)$ defined in \eqref{analyticity domain}.
In particular, $\u$ is analytic on the following domain depending on $\E$:
\begin{equation} \label{local analyticity domain}
\mathcal N_\E(\u) = \left\{ \xbold \in \mathbb R^2 \mid \dist (\E, \xbold) < c \frac{\dist(\E,\zerobold)}{d_\u}  \right\} \quad \forall c \in \left(  0, \frac{1}{2}  \right).
\end{equation}
since $\mathcal N_\E(\u) \subset \mathcal N(\u)$. This fact has an extreme relevance in the proof of forthcoming Lemma \ref{lemma approximation harmonic polynomials far from singularity}.
The important issue is that more the polygon is near the singularity, the smaller is the extended domain $\mathcal N_\E(\u)$, see Figure \ref{figure extension gets smaller}.

\begin{figure}  [h]
\begin{center}
\begin{minipage}{0.30\textwidth}
\begin{tikzpicture}[scale=0.4]
\draw[black, very thick, -] (0,0) -- (0,-4) -- (4,-4) -- (4,4) -- (-4,4) -- (-4,0) -- (0,0);
\draw[black, very thick, -] (0,0) -- (4,0); \draw[black, very thick, -] (0,0) -- (0,4);
\draw[black, very thick, -] (-4,2) -- (4,2); \draw[black, very thick, -] (2,-4) -- (2,4); \draw[black, very thick, -] (-2,0) -- (-2,4); \draw[black, very thick, -] (0,-2) -- (4,-2);
\draw(-3,1) node[black] {\tiny{{$\E$}}};
\draw[red, very thick, -] (-5,-1) -- (-1,-1) -- (-1,3) -- (-5,3) -- (-5,-1);
\draw(-3,-1.5) node[black] {\tiny{{$\mathcal N_\E(\u)$}}};
\end{tikzpicture}
\end{minipage}
\begin{minipage}{0.30\textwidth}
\begin{tikzpicture}[scale=0.4]
\draw[black, very thick, -] (0,0) -- (0,-4) -- (4,-4) -- (4,4) -- (-4,4) -- (-4,0) -- (0,0);
\draw[black, very thick, -] (0,0) -- (4,0); \draw[black, very thick, -] (0,0) -- (0,4);
\draw[black, very thick, -] (-4,2) -- (4,2); \draw[black, very thick, -] (2,-4) -- (2,4); \draw[black, very thick, -] (-2,0) -- (-2,4); \draw[black, very thick, -] (0,-2) -- (4,-2);
\draw[black, very thick, -] (-2,1) -- (2,1); \draw[black, very thick, -] (1,2) -- (1,-2);
\draw[black, very thick, -] (-1,0) -- (-1,2); \draw[black, very thick, -] (0,-1) -- (2,-1);
\draw(-1.5,.5) node[black] {\tiny{{$\E$}}};
\draw[red, very thick, -] (-2.5,-.5) -- (-.5,-.5) -- (-.5,1.5) -- (-2.5,1.5) -- (-2.5,-.5);
\draw(-1.5,-0.75) node[black] {\tiny{{$\mathcal N_\E(\u)$}}};
\end{tikzpicture}
\end{minipage}
\begin{minipage}{0.33\textwidth}
\begin{tikzpicture}[scale=0.4]
\draw[black, very thick, -] (0,0) -- (0,-4) -- (4,-4) -- (4,4) -- (-4,4) -- (-4,0) -- (0,0);
\draw[black, very thick, -] (0,0) -- (4,0); \draw[black, very thick, -] (0,0) -- (0,4);
\draw[black, very thick, -] (-4,2) -- (4,2); \draw[black, very thick, -] (2,-4) -- (2,4); \draw[black, very thick, -] (-2,0) -- (-2,4); \draw[black, very thick, -] (0,-2) -- (4,-2);
\draw[black, very thick, -] (-2,1) -- (2,1); \draw[black, very thick, -] (1,2) -- (1,-2);
\draw[black, very thick, -] (-1,0) -- (-1,2); \draw[black, very thick, -] (0,-1) -- (2,-1);
\draw[black, very thick, -] (-1,0.5) -- (1,0.5); \draw[black, very thick, -] (0.5,1) -- (0.5,-1);
\draw[black, very thick, -] (-0.5,0) -- (-0.5,1); \draw[black, very thick, -] (0,-0.5) -- (1,-0.5);
\draw(-1.5/2,.5/2) node[black] {\tiny{{$\E$}}};
\draw[red, very thick, -] (-2.5/2,-.5/2) -- (-.5/2,-.5/2) -- (-.5/2,1.5/2) -- (-2.5/2,1.5/2) -- (-2.5/2,-.5/2);
\draw(-1.5/2,-0.75/2) node[black] {\tiny{{$\mathcal N_\E(\u)$}}};
\end{tikzpicture}
\end{minipage}
\end{center}
\caption{Given $\E$ polygon in $\taun$, its extension keeps separated from the singularity, since the smaller is the polygon the smaller can be taken the extension.}
\label{figure extension gets smaller}
\end{figure}
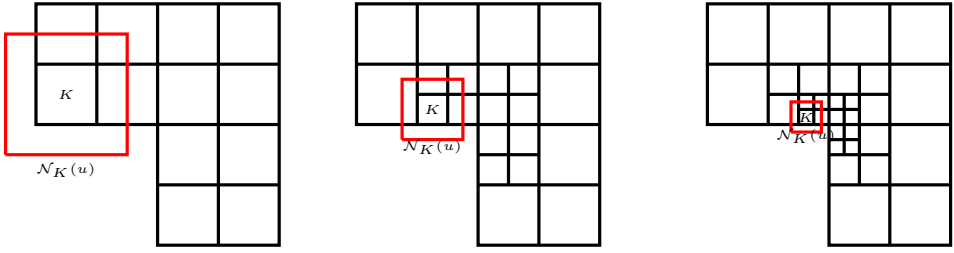

In any case, $\mathcal N_\E(\u)$ remains contained in the global analyticity domain $\mathcal N(\u)$, which is fixed once and for all.

We choose $c = \frac{1}{4}$ in \eqref{local analyticity domain}.
Owing to \eqref{assumption geometric graded} and recalling that $\E \notin L_0$, there exist two constants $0<\alpha_1\le 1 \le \alpha_2$ independent of $\E$ such that $\alpha_1 h_\E \le \dist (\E, \mathbf 0) \le \alpha_2 h_\E$. Thus,
\[
\frac{1}{4} \frac{\dist(\E,\mathbf{0})}{d_\u} = \frac{1}{4} \alpha_1 \frac{\alpha_1^{-1}\dist(\E,\mathbf{0})}{d_\u} \ge \frac{1}{4} \frac{\alpha_1}{d_\u} h_\E.
\]
This implies that $\u$ is analytic on the following domain too:
\begin{equation} \label{reduced local analyticity domain}
\widetilde {\mathcal N} _\E(\u) = \left\{ \xbold \in \mathbb R^2 \mid \dist (\E,\xbold) < \frac{1}{4}  \frac{\alpha_1}{d_\u} h_\E   \right\}\subseteq \mathcal N_\E(\u),\quad \E \in L_j,\; j=1,\dots, n.
\end{equation}
Therefore, we fix for instance $\varepsilon = \frac{1}{8} \frac{\alpha_1}{d_\u}h_\E$. In this way, we have built $\E_\varepsilon = \widetilde {\mathcal N}_\E(\u)$ neighbourhood of $\E$ not covering $\mathbf 0$.

It is straightforward to note that scaling $\E$ to $\Ehat$ with $h_{\Ehat} = 1$, we also scale $\E _{\varepsilon}$ to $\Ehat_{\delta}$ (see \eqref{extension polygon} for the definition of $\Ehat_{\delta}$),
where $\delta = \frac{1}{8} \frac{\alpha_1}{d_\u}$ is now independent of $\E$ and only depends on $\u$.
Fixing such a $\delta$ in Theorem \ref{theorem HMPS21}, we have that \eqref{same inequality as above for variable diameter} holds with $\mathbf 0 \notin \overline{\E_\varepsilon}$;
in particular, the norm appearing on the right-hand side of \eqref{same inequality as above for variable diameter} is finite for all $\E \in L_j$, $j=1,\dots,n$.

We are now ready to state the bound on the best error with respect to harmonic polynomials on the polygons not abutting the singularity.
\begin{lem} \label{lemma approximation harmonic polynomials far from singularity}
Let the assumptions (\textbf{D1})-(\textbf{D3}) hold true, let $\E \in L_j$, $j=1,\dots,n$ and let $\u \in W^{1,\infty}(\widetilde  {\mathcal N}_\E(\u))$,
where $\widetilde {\mathcal N}_\E(\u)$ is defined in \eqref{reduced local analyticity domain}. Then, there exists a sequence $\{\q_{\pE}\}_{\pE=1}^{\infty} \subseteq \{\mathbb H_\p(\E)\}_{\p=1}^{\infty}$ of harmonic polynomials such that
\begin{equation} \label{final estimate harmonic far from singularity}
\vert \u -\q_{\pE} \vert _{1,\E} \lesssim \h_{\widetilde  {\mathcal N}_\E(\u)}\exp{(-b\pE)} \Vert \u \Vert_{W^{1,\infty} (\widetilde {\mathcal N}_\E(\u))} \lesssim \exp{(-b\pE)},
\end{equation}
where $b$ is a constant independent of $\E$.
\end{lem}
\begin{proof}
The proof follows from Theorem \ref{theorem HMPS21} and the subsequent discussion.
In particular, the first inequality in \eqref{final estimate harmonic far from singularity} follows from a scaling argument, whereas, the second one is a consequence of computations
analogous to those in \eqref{computations series} and the definition of $\widetilde {\mathcal N}_\E(\u)$ in \eqref{reduced local analyticity domain}.
\end{proof}
\medskip

It is clear from the above discussion that we must follow a different strategy for the elements in the first layer; in fact, here, the $W^{1,\infty}$ norm of $\u$ is not finite in principle.

It holds in particular the following result.
\begin{lem} \label{lemma harmonic polynomials near singularity}
Let the assumptions (\textbf{D1})-(\textbf{D3}) hold true. Let $\E \in L_0$. Let $\u \in H^{2,2}_{\beta}(\O)$. Then, there exists $\q_1 \in \mathbb P_{1}(\E)$ such that
\[
\vert \u - \q_1 \vert_{1,\E} \lesssim h_\E^{2(1-\beta)} \Vert \vert \xbold^{\beta} \vert \vert \D^2 \u \vert \Vert_{0,\E}^2 \lesssim \sigma^{2(1-\beta)n}.
\]
In particular, $\q_1$ is a harmonic polynomial.
\end{lem}
\begin{proof}
The polynomial $\q_1$ is given by the linear interpolant of $\u$ at, for instance, three nonaligned vertices of \E. The proof follows the lines of \cite[Lemma 3]{hpVEMcorner}.
\end{proof}

\begin{remark} \label{remark polynomial approx near sing}
Lemma \ref{lemma harmonic polynomials near singularity} suggests that one could also consider harmonic virtual element spaces with nonuniform degrees of accuracy, still guaranteeing exponential convergence for the $\h\p$ version of the method.
In particular, one could consider a distribution of degrees of accuracy which grows linearly as the distance from the singularity increases, as depicted in Figure \ref{figure non uniform degree of accuracy}.

\begin{figure}  [h]
\begin{center}
\begin{minipage}{0.30\textwidth}
\begin{tikzpicture}[scale=0.4]
\draw[black, very thick, -] (0,0) -- (0,-4) -- (4,-4) -- (4,4) -- (-4,4) -- (-4,0) -- (0,0);
\draw[black, very thick, -] (0,-2) -- (2,-2) -- (2,2) -- (-2,2) -- (-2,0);
\draw[black, very thick, -] (0,0) -- (4,4);
\draw(-3.4,3.4) node[black] {{$2$}}; \draw(3.4, -3.4) node[black] {{$2$}};
\draw(-1.4,1.4) node[black] {{$1$}}; \draw(1.4, -1.4) node[black] {{$1$}};
\end{tikzpicture}
\end{minipage}
\begin{minipage}{0.30\textwidth}
\begin{tikzpicture}[scale=0.4]
\draw[black, very thick, -] (0,0) -- (0,-4) -- (4,-4) -- (4,4) -- (-4,4) -- (-4,0) -- (0,0);
\draw[black, very thick, -] (0,-2) -- (2,-2) -- (2,2) -- (-2,2) -- (-2,0);
\draw[black, very thick, -] (0,-1) -- (1,-1) -- (1,1) -- (-1,1) -- (-1,0);
\draw[black, very thick, -] (0,0) -- (4,4);
\draw(-3.4,3.4) node[black] {{$3$}}; \draw(3.4, -3.4) node[black] {{$3$}};
\draw(-1.4,1.4) node[black] {{$2$}}; \draw(1.4, -1.4) node[black] {{$2$}};
\draw(-.4,.4) node[black] {{$1$}}; \draw(.4, -.4) node[black] {{$1$}};
\end{tikzpicture}
\end{minipage}
\begin{minipage}{0.33\textwidth}
\begin{tikzpicture}[scale=0.4]
\draw[black, very thick, -] (0,0) -- (0,-4) -- (4,-4) -- (4,4) -- (-4,4) -- (-4,0) -- (0,0);
\draw[black, very thick, -] (0,-2) -- (2,-2) -- (2,2) -- (-2,2) -- (-2,0);
\draw[black, very thick, -] (0,-1) -- (1,-1) -- (1,1) -- (-1,1) -- (-1,0);
\draw[black, very thick, -] (0,-.5) -- (.5,-.5) -- (.5,.5) -- (-.5,.5) -- (-.5,0);
\draw[black, very thick, -] (0,0) -- (4,4);
\draw(-3.4,3.4) node[black] {{$4$}}; \draw(3.4, -3.4) node[black] {{$4$}};
\draw(-1.4,1.4) node[black] {{$3$}}; \draw(1.4, -1.4) node[black] {{$3$}};
\draw(-0.67,0.67) node[black] {{\tiny{$2$}}}; \draw(0.7, -0.7) node[black] {\tiny{{$2$}}};
\draw(-0.2,0.2) node[black] {{\tiny{$1$}}}; \draw(0.2, -0.2) node[black] {\tiny{{$1$}}};
\end{tikzpicture}
\end{minipage}
\end{center}
\caption{Nonuniform distribution of degrees of accuracy. In layer $L_0$, $\p=1$. In layers $L_j$, $j=1,\dots,n$, $\p\in \mathbb N$.}
\label{figure non uniform degree of accuracy}
\end{figure}
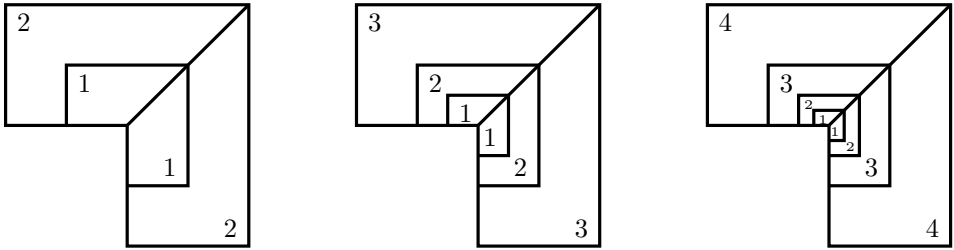
At the interface $\e$ of two nondisjoint elements $\E_0$ and $\E_1$ in layers $L_0$ and $L_1$ one associates $\p_\e = \max(1,\p) = \p$ (\emph{maximum rule})
in order to define nonuniform boundary spaces $\mathbb B(\partial \E)$ similarly to \eqref{boundary space}, as depicted in Figure \ref{figure maximum rule}.
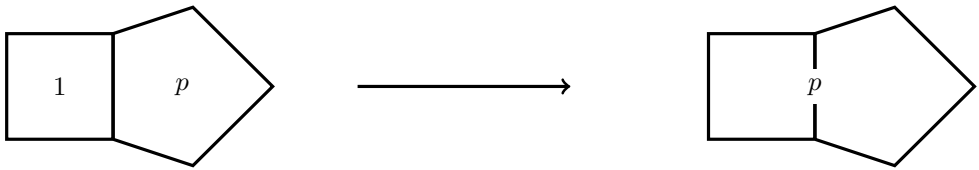
\begin{figure}  [h]
\begin{center}
\begin{minipage}{0.30\textwidth}
\begin{tikzpicture}[scale=0.7]
\draw[black, very thick, -] (0,-1) -- (0,1) -- (-2,1) -- (-2,-1) -- (0,-1);
\draw[black, very thick, -] (0,-1) -- (1.5,-3/2) -- (3,0) -- (1.5,3/2) -- (0,1) -- (0,-1);
\draw(-1,0) node[black] {{$1$}};
\draw( 1.3,0) node[black] {{$\p$}};
\end{tikzpicture}
\end{minipage}
\begin{minipage}{0.30\textwidth}
\begin{tikzpicture}[scale=0.7]
\draw[black, very thick, ->] (-2,0) -- (2,0);
\end{tikzpicture}
\end{minipage}
\begin{minipage}{0.30\textwidth}
\begin{tikzpicture}[scale=0.7]
\draw[black, very thick, -] (0,-1) -- (0,1) -- (-2,1) -- (-2,-1) -- (0,-1);
\draw[black, very thick, -] (0,-1) -- (1.5,-3/2) -- (3,0) -- (1.5,3/2) -- (0,1) -- (0,-1);
\draw(0,0) node[black, fill=white] {{$\p$}};
\end{tikzpicture}
\end{minipage}
\end{center}
\caption{If one considers nonuniform degrees of accuracy, then the largest polynomial degree at the interface can be taken (\emph{maximum rule}).
}
\label{figure maximum rule}
\end{figure}
\end{remark}

In this section, we have thus built a piecewise discontinuous harmonic polynomial with certain approximation properties described in Lemmata \ref{lemma approximation harmonic polynomials far from singularity} and \ref{lemma harmonic polynomials near singularity}.
Such a discontinuous function will be used in the proof of Theorem \ref{theorem exponential convergence} in the approximation of the first term on the right-hand side of \eqref{two terms splitting}.
\subsection{Approximation by functions in the harmonic virtual element space} \label{subsection approximation by functions in the Harmonic Element Space}
Here, we discuss about approximation estimates by functions in the harmonic virtual element space which will be used for the approximation of the second term in \eqref{two terms splitting}.
As in Section \ref{subsection approximation by harmonic polynomials}, we firstly investigate approximation estimates on polygons not abutting the singularity, see Lemma \ref{lemma harmonic approximation far from sing};
secondly, we discuss approximation estimates of polygons in the finest layer $L_0$, see Lemma \ref{lemma harmonic approximation near sing}.
\begin{lem} \label{lemma harmonic approximation far from sing}
Let the assumptions (\textbf{D1})-(\textbf{D3}) hold true. Let $\E \in L_j$, $j=1,\dots, n$ and let $\beta \in [0,1)$.
Let $\g$, the Dirichlet datum of problem \eqref{weak continuous problem}, belong to space $\mathcal B^{\frac{3}{2}}_\beta (\partial \O)$
and let  $\u$, the solution to problem \eqref{weak continuous problem}, belong to space $\mathcal B^{2}_{\beta}(\O)$, see \eqref{countably normed spaces}.
Then, there exists $\uI \in \VE$ such that
\[
\begin{split}
\vert \u - \uI \vert_{1,\E} 	& \lesssim e^{\s+\frac{1}{2}} \left( \frac{\hE}{\pE} \right)^{\s+\frac{1}{2}} \left( \sum_{\e \in \mathcal E^\E} \vert u \vert^2_{m+1,\e}   \right)^{\frac{1}{2}}\\
				& \lesssim e^{\s+\frac{1}{2}} \p^{-\s - \frac{1}{2}} \sigma^{(n-j)(1-\beta)}  \left\{ \vert \u \vert_{H^{\s+1,2}_\beta (\E)} + \vert \u \vert_{H^{\s+2,2}_\beta (\E)}  \right\}\quad \forall \,\s \in \mathbb N_0,
\end{split}
\]
where we recall that $\sigma$ is the geometric grading parameter of the assumption (\textbf{D3}).
\end{lem}
\begin{proof}
Before proving the result, we observe that the $H^{\sE+1}(\e)$ seminorm exists for all edges $\e$ of $\E$, since $\u \in \mathcal B _\beta^2(\O)$ implies that $\u$ is analytic far from the singularity.

Let us consider $\uI \in \VE$ defined as the weak solution to the following local Laplace problem:
\begin{equation} \label{harmonic interpolant}
\begin{cases}
-\Delta \uI = 0 & \text{in } \E\\
\uI = \uGL & \text{on } \partial \E,\\
\end{cases}
\end{equation}
where $\uGL$ is the Gau\ss-Lobatto interpolant of degree $\pe$ of $\u$ on each edge $\e$. Then, using the fact that $\u-\uI$ is harmonic and using a Neumann trace inequality \cite[Theorem A.33]{SchwabpandhpFEM}, one gets
\begin{equation} \label{estimate with the trick of the c}
\begin{split}
\vert \u - \uI \vert^2_{1,\E} = \int_{\partial \E} \partial _\n(\u - \uI) (\u - \uI-c) 	& \le \left\Vert  \partial_\n (\u - \uI)  \right\Vert_{-\frac{1}{2},\partial \E} \Vert \u - \uGL -c\Vert_{\frac{1}{2}, \partial \E}\\
																& \lesssim \vert \u -\uI \vert_{1,\E} \Vert \u - \uGL -c \Vert_{\frac{1}{2}, \partial \E},
\end{split}
\end{equation}
for every $c \in \mathbb R$.

We deduce that we must deal with the boundary error term only.  We fix $c=0$ in \eqref{estimate with the trick of the c} (the case $c \neq 0$ will become important in the following).
Since $\u$ is analytic far from the singularity, we inherit the two following results from \cite[Theorems 4.2 and 4.5]{bernardimaday1992polynomialinterpolationinsobolev}:
\[
\begin{split}
& \Vert \u - \uGL \Vert_{0,\e} \lesssim e^{\sE+1} \left( \frac{h_\e}{\pe} \right) ^{\sE+1}  \vert \u \vert_{\sE+1,\e},\quad \forall \e \text{ edge of }\E, \; \forall \sE \in \mathbb N_0,\\
& \vert \u - \uGL \vert_{1,\e} \lesssim e^{\sE} \left( \frac{h_\e}{\pe} \right) ^{\sE}  \vert \u \vert_{\sE+1,\e} ,\quad \forall \e \text{ edge of }\E, \; \forall \sE \in \mathbb N_0.\\
\end{split}
\]
Using interpolation theory \cite{Triebel}, recalling from the assumption (\textbf{D2}) that $h_\e \approx h_\E$ and that the number of edges of each $\E \in \taun$ is uniformly bounded, yield
\begin{equation} \label{Bernardi Maday argument}
\Vert \u - \uI \Vert ^2_{\frac{1}{2},\partial \E} = \Vert \u - \uGL \Vert ^2_{\frac{1}{2},\partial \E} \lesssim e^{2\sE+1} \left( \frac{h_\E}{\pE}   \right)^{2\sE+1} \sum_{\e \in \mathcal E^\E} \vert \u \vert_{\sE+1,\e}^2. 
\end{equation}
We apply a multiplicative trace inequality \cite[Theorem 1.6.6]{BrennerScott},
the fact that the maximum number of edges of $\E$ is uniformly bounded, see the assumption (\textbf{D2}), and the trivial bound $|a||b| \le a^2+b^2$, $a$, $b\in \mathbb R$, getting
\begin{equation} \label{scaling multiplicative trace inequality}
\sum_{\e\in \mathcal E^\E} \vert \u \vert^2_{\s+1, \e} \lesssim \hE^{-1} \vert \u \vert^2_{\sE+1,\E} + \hE\vert \u \vert ^2 _{\sE+2,\E}.
\end{equation}
Recalling the definition of the weighted Sobolev seminorms \eqref{weighted Sobolev norms}, one obtains
\begin{equation} \label{deducing Babuska seminorms}
\vert \u \vert^2_{H_\beta^{\s+\ell,2}(\E)} \ge \Vert \Phi_{\beta + \s + \ell - 2} \,\vert \D^{(\s+\ell)}\u \vert \,\Vert_{0,\E}^2 \gtrsim \dist(\E,\mathbf 0) ^{2(\beta + \s +\ell - 2)} \vert \u \vert^2_{\s+\ell,\E},\quad \ell =1,2.
\end{equation}
Combining \eqref{assumption geometric graded}, \eqref{scaling multiplicative trace inequality} and \eqref{deducing Babuska seminorms}, we deduce
\begin{equation} \label{Babuska estimate trace}
\vert \u \vert^2_{\s + 1, \partial \E} \lesssim \hE^{-2 (\beta + \s - \frac{1}{2})} \left\{  \vert \u \vert^2_{H_\beta^{\s+1,2}(\E)} + \vert \u \vert^2_{H_\beta^{\s+2,2}(\E)} \right\}.
\end{equation}
Finally, recalling from the assumption~(\textbf{D3}) that $\hE \approx \sigma^{n-j}$, we get the claim by inserting~\eqref{Babuska estimate trace} in~\eqref{Bernardi Maday argument}.
\end{proof}

Next, we turn our attention to the approximation in the polygons belonging to the first layer.
\begin{lem} \label{lemma harmonic approximation near sing}
Let the assumptions (\textbf{D1})-(\textbf{D3}) hold true. Let $\E \in L_0$ and let $\beta\in [0,1)$.
Let $\g$, the Dirichlet datum of problem \eqref{weak continuous problem}, belong to space $\mathcal B^{\frac{3}{2}}_\beta (\partial \O)$ and let
$\u$, the solution to problem \eqref{weak continuous problem}, belong to space $\mathcal B^{2}_{\beta}(\O)$ \eqref{countably normed spaces}.
Then, there exists $\uI \in \VE$ such that
\[
\vert \u - \uI \vert^2_{1,\E} \lesssim \sigma^{2n(1-\beta)},
\]
where we recall that $\sigma$ is the geometric grading parameter of the assumption (\textbf{D3}).
\end{lem}
\begin{proof}
Let $\uI$ be defined as in \eqref{harmonic interpolant}, with $\uGL$ being now the linear interpolant of $\u$ on each edge $\e$ of $\E$.
Let $\tautilden(\E)$ be the subtriangulation of $\E$ obtained by joining $\mathbf 0$ with the other vertices of $\E$. Such a subtriangulation is regular, see the assumption (\textbf{D1}).

From \eqref{estimate with the trick of the c}, we have
\[
\vert \u - \uI \vert_{1,\E} \lesssim \Vert \u - \uGL -c \Vert_{\frac{1}{2},\partial \E} \quad \forall c \in \mathbb R.
\]
We denote by $\utildeGL$ the linear interpolant of $\u$ over every $\T \in \tautilden(\E)$ at the three vertices of $\T$.
One obviosuly has $\utildeGL = \uGL$ on $\partial \E$. Applying a trace inequality, we get
\[
\vert \u - \uI \vert_{1,\E} \lesssim \Vert \u - \utildeGL - c \Vert_{1,\E}.
\]
By picking $c$ the average of $\u - \utildeGL$ over $\E$, applying a Poincar\'e inequality and recalling that $\card(\tautilden)$ is uniformly bounded, we deduce
\[
\vert \u - \uI \vert _{1,\E}^2 \lesssim \sum_{\E \in \tautilden(\E)} \vert \u - \utildeGL \vert^2_{1,\T}.
\]
In order to conclude, we apply \cite[Lemma 4.16]{SchwabpandhpFEM} and \eqref{assumption geometric graded} obtaining
\[
\vert \u - \uI \vert^2_{1,\E} \lesssim \sum_{\E \in \tautilden(\E)} h_\T^{2(1-\beta)} \Vert \vert \xbold \vert ^\beta \vert D^2 \u \vert \Vert^2_{0,\T}
\lesssim \sigma ^{2n(2-\beta)} \Vert \vert \xbold \vert ^\beta \vert D^2 \u \vert \Vert^2_{0,\T} \lesssim \sigma^{2n(1-\beta)},
\]
which holds true owing to the fact that $\u \in \mathcal B^2_\beta (\O)$.
\end{proof}

Again, for the proof of Lemma \ref{lemma harmonic approximation near sing}, one could have used nonuniform degrees of accuracy as discussed in Remark \ref{remark polynomial approx near sing}.


In order to conclude this section, we highlight that we built in Lemmata \ref{lemma harmonic approximation far from sing} and \ref{lemma harmonic approximation near sing} a continuous approximant
of $\u$, which belongs to space $\Vng$ \eqref{global harmonic space}.

\subsection*{The $\h$ version of harmonic VEM for quasi-uniform meshes.}
Although the goal of this paper is to study the $\h\p$ version of harmonic VEM, it is worthwhile to mention that the $\h$ version of the method employing sequences of quasi-uniform meshes, see e.g. \cite{VEMvolley} for the definition of quasi-uniform meshes,
easily follows by combining Lemma \ref{lemma Strang harmonic}, \cite[Theorem 2]{babumelenk_harmonicpolynomials_approx} and Lemma \ref{lemma harmonic approximation far from sing}.

In particular, assuming that $u$, the solution to problem \eqref{weak continuous problem}, belongs to $H^{s+1}(\Omega)$, $s\in \mathbb R_+$,
and that we employ harmonic virtual element spaces with a uniform degree of accuracy $\p$, one gets
\begin{equation}\label{h version}
\vert \u - \un \vert_{1,\Omega} \lesssim \h^{\min(s,\p)} \Vert \u \Vert_{s+1,\Omega},
\end{equation}
where the hidden constant depends on $s$, on the shape of the elements in the mesh and the choice of the stabilization, but is independent of the mesh size $\h$.

\subsection{Exponential convergence} \label{subsection exponential convergence}
Here, we discuss the main result of the work, namely the exponential convergence of the energy error in terms of the number of degrees of freedom.
In order to achieve such a result, we fix as a degree of accuracy
\begin{equation} \label{standard choice degree of accuracy}
\p = n+1, \quad \quad \text{$n+1$ being the number of layers of $\taun$.}
\end{equation}
The main result of the paper follows.
\begin{thm} \label{theorem exponential convergence}
Let $\{\taun\}_{n\in \mathbb N_0}$ be a sequence of polygonal decomposition satisfying the assumptions (\textbf{D1})-(\textbf{D3}).
Let $\u$ and $\un$ be the solutions to problems \eqref{weak continuous problem} and \eqref{discrete problem}, respectively.
Let $\g$, the Dirichlet datum introduced in \eqref{weak continuous problem}, belong to $\mathcal B^{\frac{3}{2}}_\beta(\partial \O)$. Then,
the following holds true:
\begin{equation} \label{error decay hpHVEM}
\vert \u- \un \vert_{1,\O} \lesssim \exp{(-b \sqrt[2] N)},
\end{equation}
where $b$ is a constant independent of the discretization parameters and $N$ is the number of degrees of freedom of $\Vn$ defined in \eqref{global harmonic space}.
\end{thm}
\begin{proof}
We only give the sketch of the proof. Applying Lemma \ref{lemma Strang harmonic}, bound \eqref{explicit bound alpha},
Lemmata \ref{lemma harmonic approximation far from sing} to \ref{lemma harmonic approximation near sing} to the first term on the right-hand side of \eqref{two terms splitting} along with standard $hp$ approximation strategies \cite{SchwabpandhpFEM}
and Lemmata \ref{lemma approximation harmonic polynomials far from singularity} and \ref{lemma harmonic polynomials near singularity}
to the second term of the right-hand side of \eqref{two terms splitting} along with \cite[Theorem 5.5]{hmps_harmonicpolynomialsapproximationandTrefftzhpdgFEM}, we have
\begin{equation} \label{energy error exponential convergence number of layers}
\vert \u - \un \vert_{1,\E} \lesssim \exp{(-\widetilde b (n+1))},
\end{equation}
for some $\widetilde b$ independent of the discretization parameters, $n+1$ being the number of layers in $\taun$.

In order to conclude, it suffices to find out the relation between $n$ and $N$, the number of degrees of freedom of space $\Vn$.
To this end, we recall from \cite[(5.6)]{hmps_harmonicpolynomialsapproximationandTrefftzhpdgFEM} that in each layer $L_j$
there exists a fixed maximum number of elements, see the assumption (\textbf{D3}).
Moreover, thanks to the assumption (\textbf{D2}), there exists a fixed maximum number of edges per element.

If we set $N_{\text{edge}}$ the maximum number of edges per element and $N_{\text{element}}$ the maximum number of elements per layer, we conclude that
\[
N = \dim(\Vn) \lesssim N_{\text{edge}} N_{\text{element}} \sum_{j=0}^{n} (n+1) \lesssim (n+1)^2.
\]
In particular, $\sqrt N \lesssim n$. This, along with \eqref{energy error exponential convergence number of layers}, gives the assertion.
\end{proof}

\section {Numerical results} \label{section numerical results}
\subsection{Numerical results: $\h$ version} \label{subsection h version numerical results}
In this section, we present numerical results validating the algebraic rate of convergence of the $\h$ version of the method stated in \eqref{h version}.

To this purpose, we consider the following test case. Let $\Omega$, the domain of problem \eqref{weak continuous problem}, be the square domain
\[
\Omega = (0,1)^2
\]
and let $\u$, the solution to the problem, be
\[
\u(x,y) = \exp(x) \sin(y),
\]
which is an analytic harmonic function over $\mathbb R^2$.

Moreover, we observe that since the functions in the harmonic virtual element space are known only via their degrees of freedom,
we cannot explicitly compute the energy error.
Therefore, we study the following normalized broken $H^1$ error between $\u$ and the energy projection of $\un$:
\begin{equation} \label{computable error}
\frac{\vert \u - \Pinabla_{\pbold} \un \vert_{1,n,\O}}{\vert \u \vert_{1,\Omega}} := \frac{\sqrt{\sum_{\E \in \taun} \left \vert \u - \PinablaE_{\pE} \un \right\vert^2_{1,\E}}}{\vert \u \vert_{1,\Omega}},
\end{equation}
where $\PinablaE_{\pE}$ is defined in \eqref{nabla projector}, for all $\E\in \taun$.

Importantly, the rate of convergence of the error in \eqref{computable error} is the same as the one of the exact $H^1$ error. In order to see this, we apply a triangle inequality and the stability of the $H^1$ projection, to get
\begin{equation} \label{estimate comp err}
\vert \u - \Pinabla \un \vert_{1,n,\O} \le \vert \u - \Pinabla \u \vert_{1,n,\O} + \vert \Pinabla(\u - \un) \vert_{1,n,\Omega} \le \vert \u - \Pinabla \u \vert_{1,n,\O} + \vert \u - \un \vert_{1,n,\Omega}
\end{equation}
and after that we apply Lemma \ref{lemma Strang harmonic}, \cite[Theorem 2]{babumelenk_harmonicpolynomials_approx} and Lemma \ref{lemma harmonic approximation far from sing}.

We test the method employing sequences made of three types mesh, see Figure \ref{figure quasi uniform meshes}, namely a squares, a Voronoi-Lloyd and an hexagonal mesh.
\begin{figure}  [h]
\centering
\subfigure {\includegraphics [angle=0, width=0.32\textwidth]{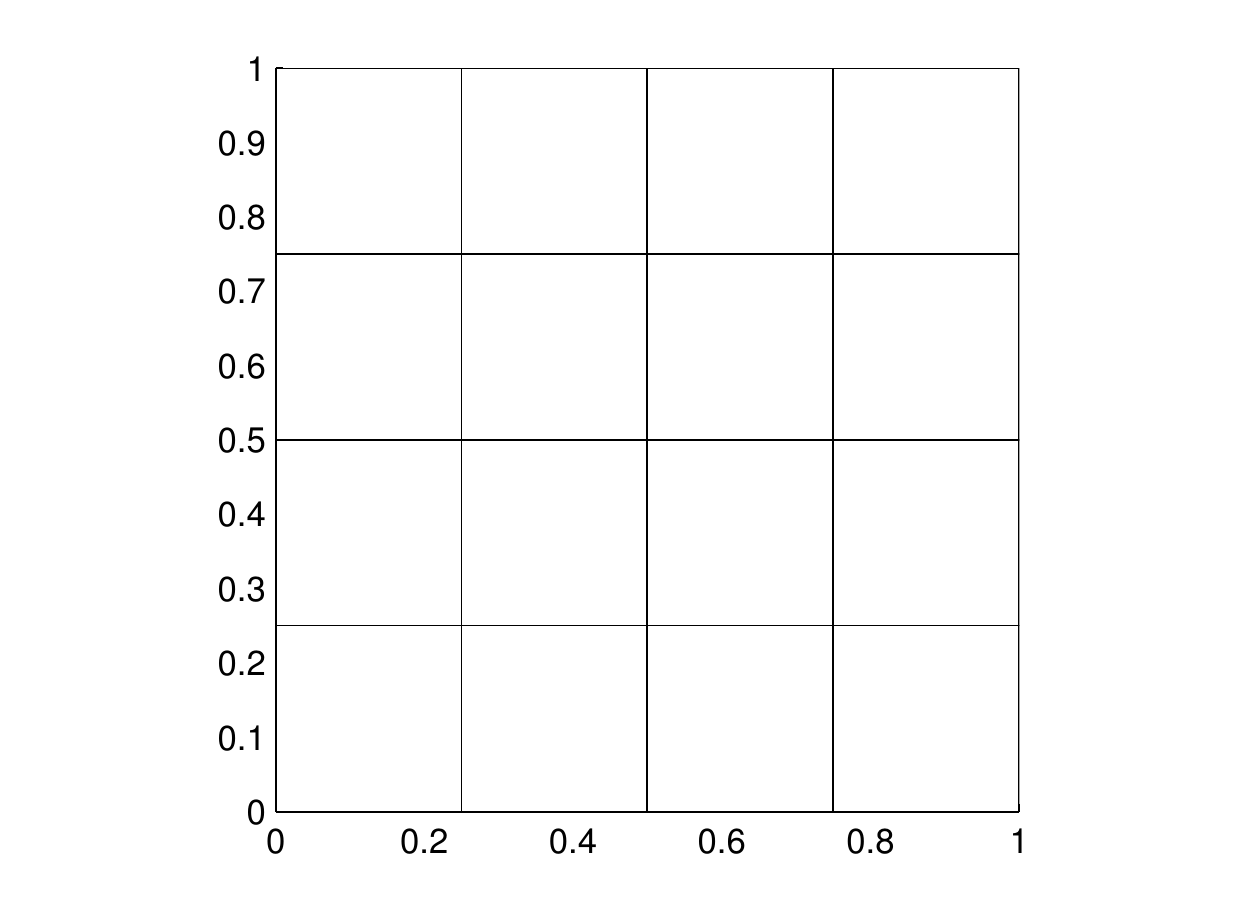}}
\subfigure {\includegraphics [angle=0, width=0.32\textwidth]{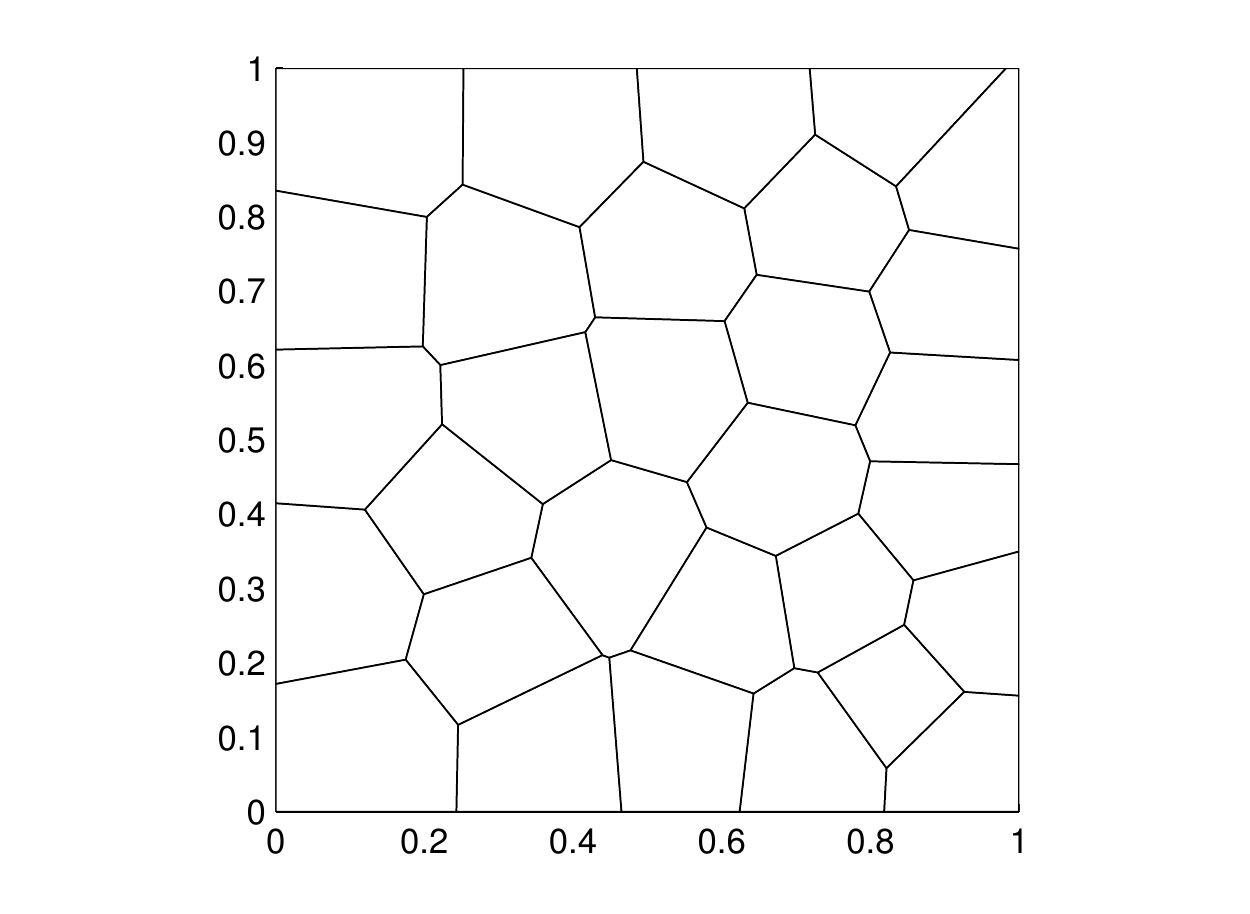}}
\subfigure {\includegraphics [angle=0, width=0.32\textwidth]{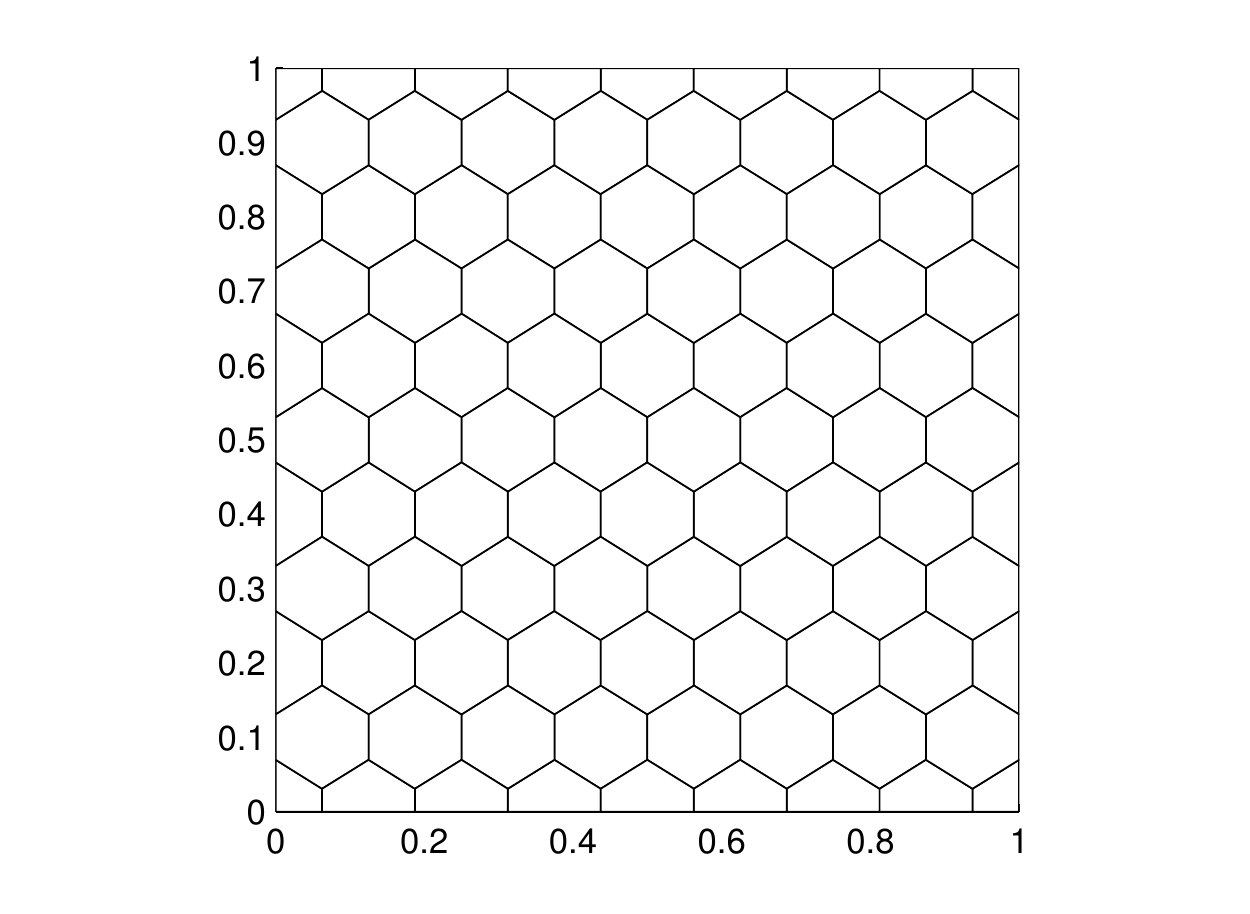}}
\caption{Left: a square mesh. Center: a Voronoi-Lloyd mesh. Right: an hexagonal mesh.}
\label{figure quasi uniform meshes}
\end{figure}
We also pick as uniform degrees of accuracy $\p=1$, $2$, $3$ and $4$.
The numerical results are collected in Figure \ref{figure nt:h} and are in accordance with the expected rate of convergence in \eqref{h version}.
\begin{figure}  [h]
\centering
\subfigure {\includegraphics [angle=0, width=0.32\textwidth]{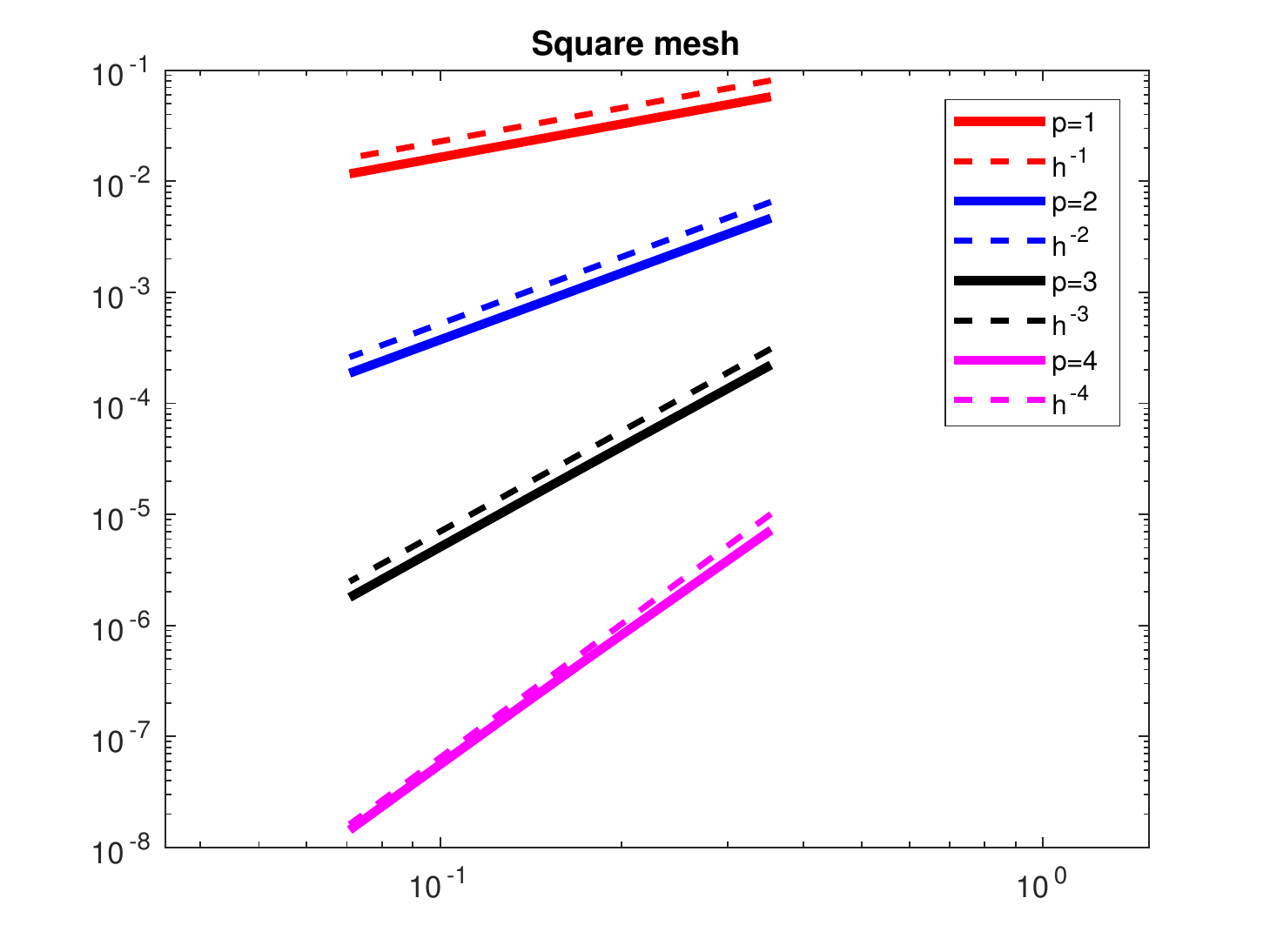}}
\subfigure {\includegraphics [angle=0, width=0.32\textwidth]{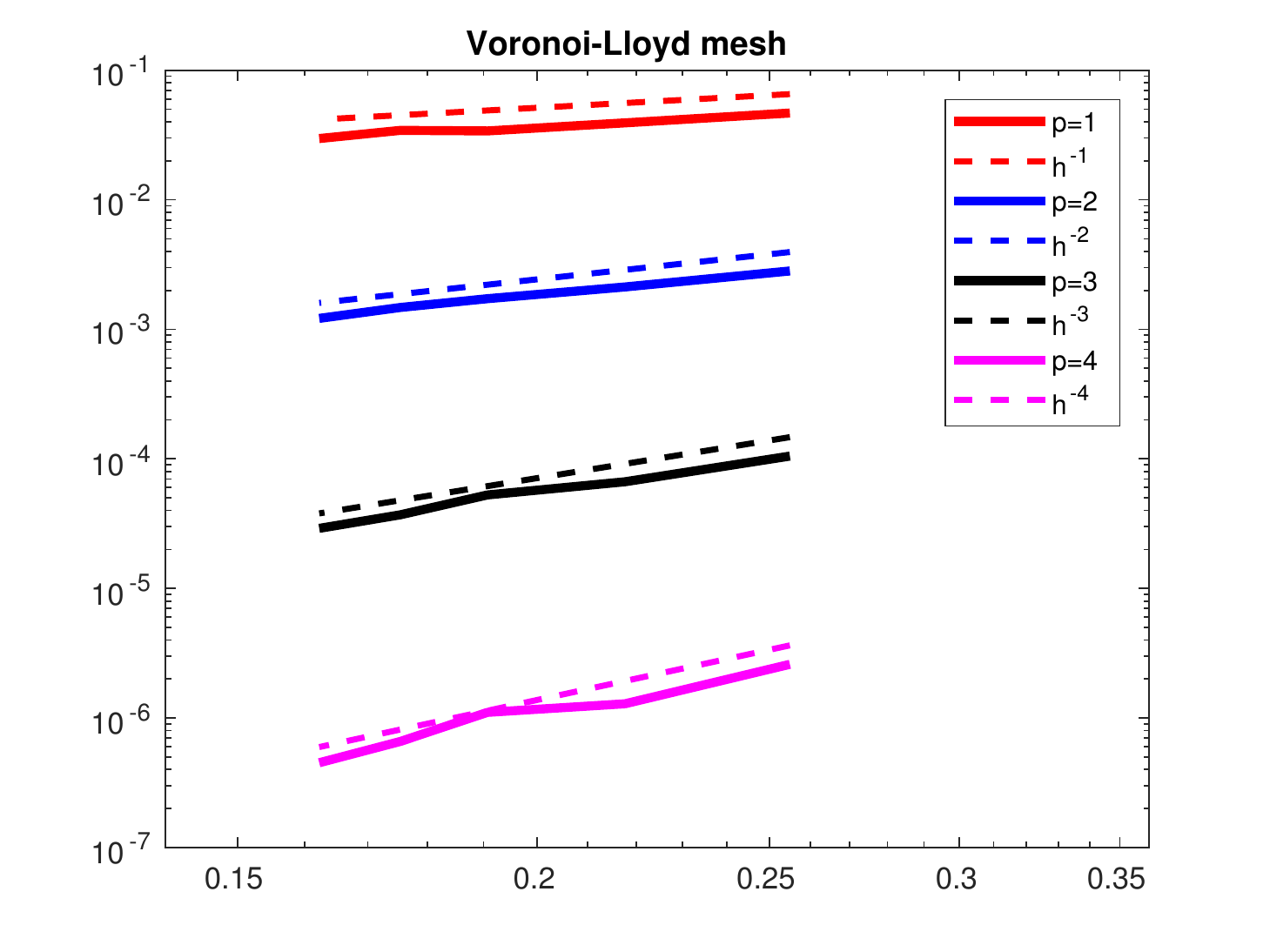}}
\subfigure {\includegraphics [angle=0, width=0.32\textwidth]{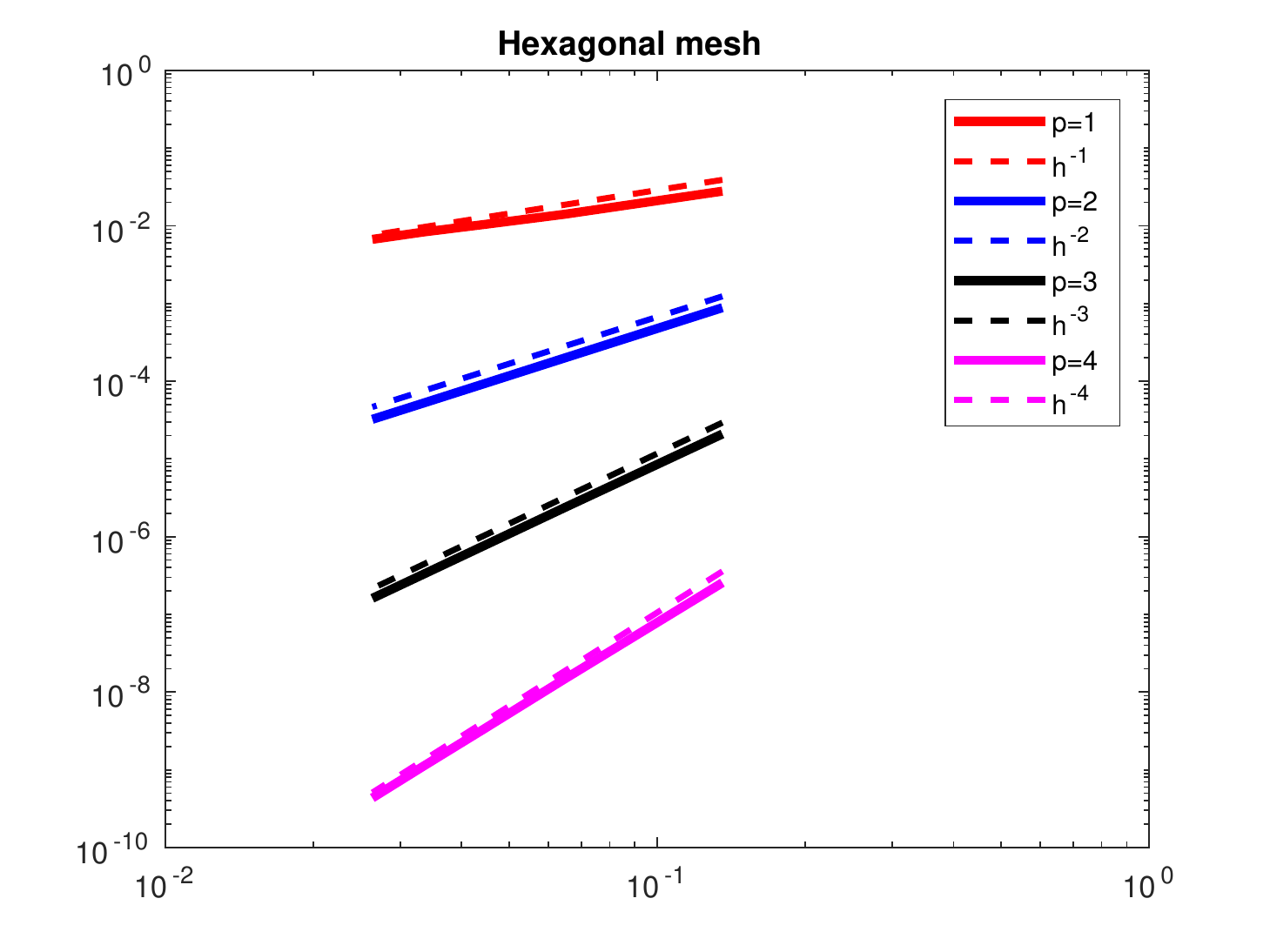}}
\caption{Error \eqref{computable error} on three sequences of meshes made of elements as those depicted in Figure \ref{figure quasi uniform meshes}.
We consider here the $\h$ version of the method. The degrees of accuracy are $\p=1$, $2$, $3$ and $4$.
Left: square meshes. Center: Voronoi-Lloyd meshes. Right: hexagonal meshes.}
\label{figure nt:h}
\end{figure}

\subsection{Numerical results: $\h\p$ version} \label{subsection hp version numerical results}
In this section, we present numerical experiments validating the exponential rate of convergence of the $\h\p$ version of the method stated in Theorem \ref{theorem exponential convergence}.
To this end, we consider the following test case. Let $\O$, the domain of problem \eqref{weak continuous problem}, be the L-shaped domain
\begin{equation} \label{domain numerical experiments}
\O = (-1,1)^2 \setminus (-1,0]^2
\end{equation}
and let $\u$, the solution to \eqref{weak continuous problem}, be
\begin{equation} \label{solution numerical experiments}
\u(r, \theta)  = r^{\frac{2}{3}} \sin \left( \frac{2}{3} \left( \theta + \frac{\pi}{2} \right) \right),
\end{equation}
where $r$ and $\theta$ are the polar coordinates of the real plane.
We observe that the such a function belongs to $H^{\frac{5}{3} - \varepsilon} (\O)$, for all $\varepsilon >0$, but not to $H^{\frac{5}{3}}(\O)$, and moreover that $\u$ is harmonic.

\subsubsection{Numerical tests on polygonal geometric graded meshes} \label{subsubsection numerical tests on polygonal geometric graded meshes}
We consider sequences of meshes as those depicted in Figure~\ref{figure possible geometric decompositions} and we consider two different distributions  of local degrees of accuracy.

We firstly investigate in Figure \ref{figure numerical test p uniform} the performances of the harmonic VEM choosing
a distribution of degrees of accuracy $\p$ as in \eqref{standard choice degree of accuracy}.
Under this choice, we know that Theorem \ref{theorem exponential convergence} holds true.

Secondly, we investigate in Figure \ref{figure numerical test p mu graded} the performances of the harmonic VEM by taking a nonuniform distribution of degrees of accuracy. In particular, we consider the (graded) distribution given by
\begin{equation} \label{graded degrees of accuracy}
\p_\E = j+1, \quad \quad 	\text{where} \quad \quad \E \in L_j,\quad j=0,\dots,n.
\end{equation}
At the interface of two polygons in different layers one associate a polynomial degree~$\p_\e$ via the \emph{maximum rule} as in Figure~\ref{figure maximum rule}, thus modifying straightforwardly the definition of the space $\mathbb B(\partial \E)$
defined in~\eqref{boundary space}.
It is worth to notice that under choice \eqref{graded degrees of accuracy} the dimension of space $\Vn$ is asymptotically $\frac{1}{2} n^2$, $n+1$ being the number of layers.
Such a dimension is comparable with the one of space $\Vn$ assuming~\eqref{standard choice degree of accuracy}, which is asymptotically $n^2$.

In both figures, we consider sequences of meshes with different geometric refinement parameters $\sigma$;
we recall that the properties fulfilled by $\sigma$ are discussed in the assumption (\textbf{D3}).
We fix in particular $\sigma=\frac{1}{2}$, $\sigma=\sqrt 2 -1$ and $\sigma=(\sqrt 2 -1)^2$.

As observed already in Section \ref{subsection h version numerical results}, we study the computable error in \eqref{computable error} in lieu of the exact one,
whose convergence in terms of the number of degrees of freedom is the same as the one of the exact $H^1$ error.

On the $y$-axis we consider the logarithm of the error defined in \eqref{computable error}, while in the $x$-axis we put the square root of the number of degrees of freedom.

\begin{figure}  [h]
\centering
\subfigure {\includegraphics [angle=0, width=0.32\textwidth]{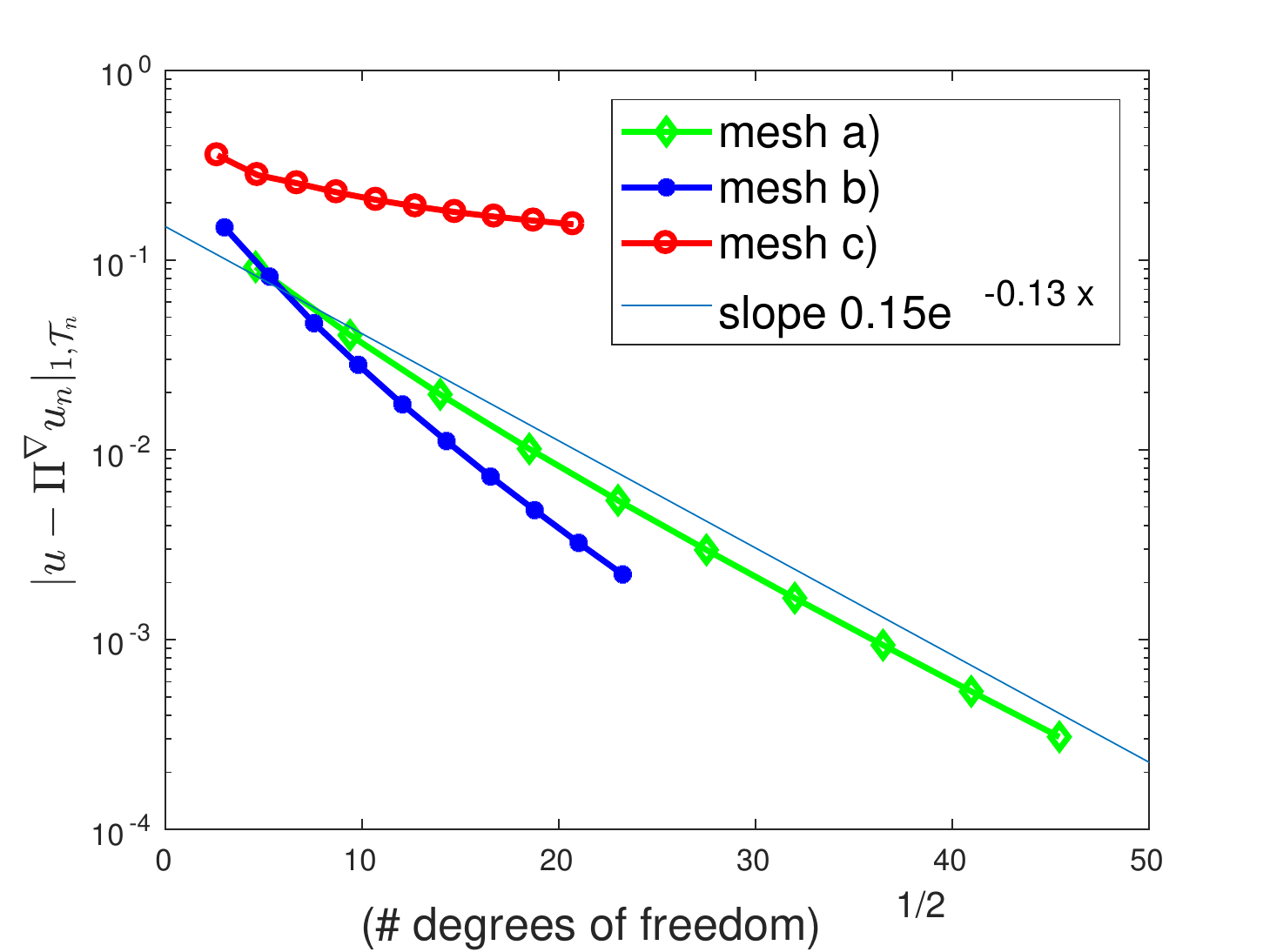}}
\subfigure {\includegraphics [angle=0, width=0.32\textwidth]{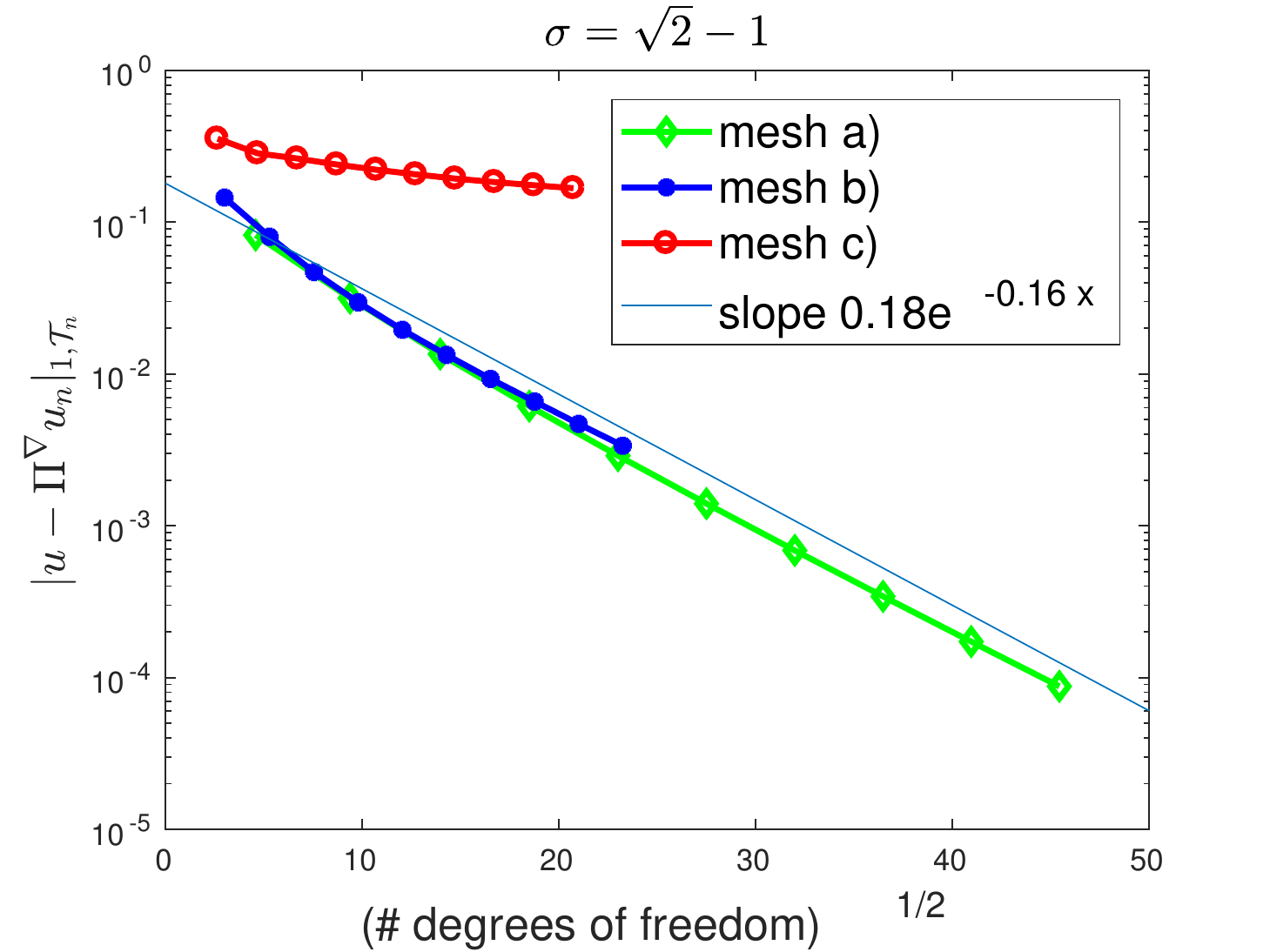}}
\subfigure {\includegraphics [angle=0, width=0.32\textwidth]{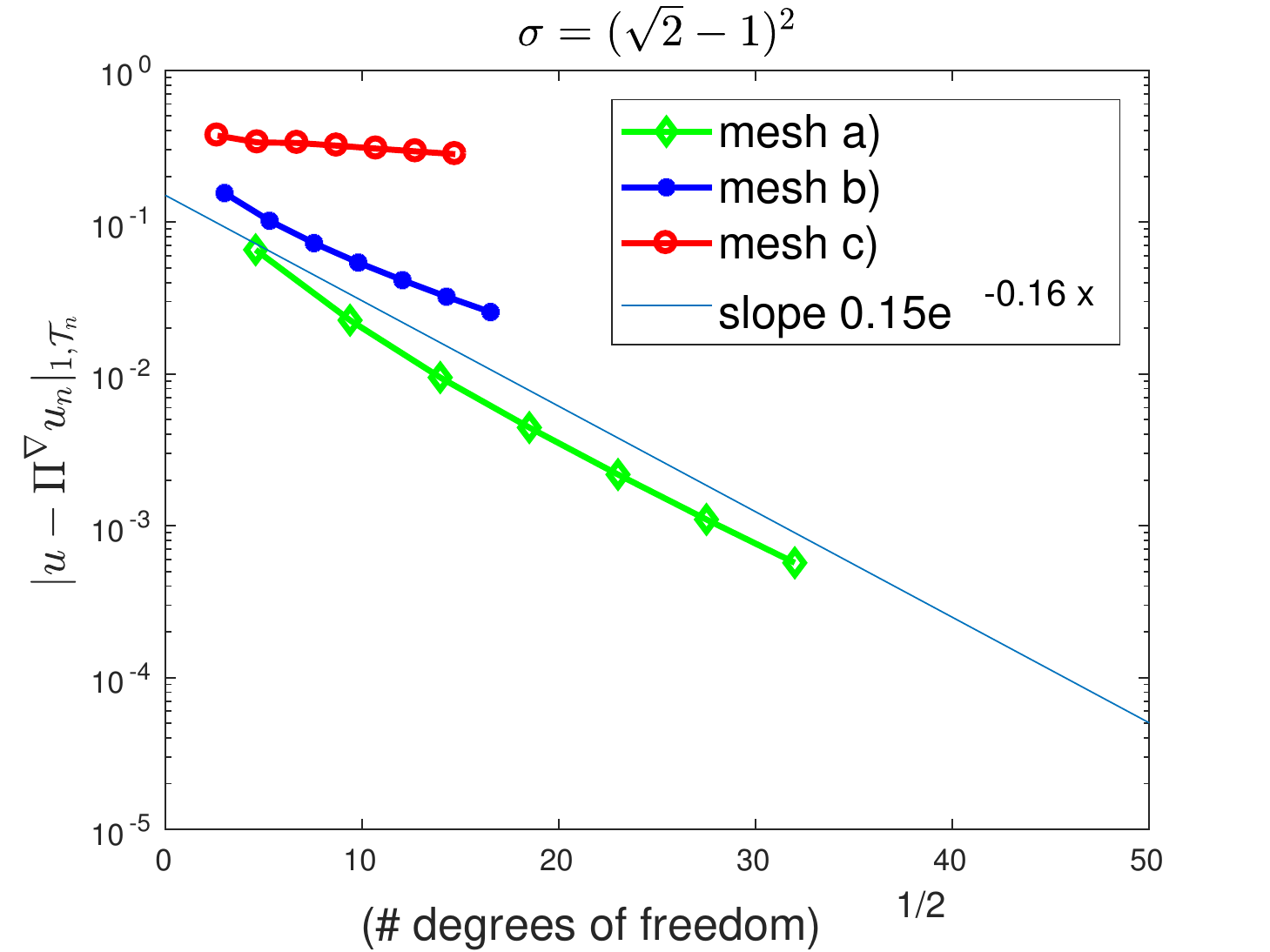}}
\caption{Error \eqref{computable error} on the three sequences of graded meshes made of elements as those in Figure \ref{figure possible geometric decompositions}.
We consider here the $\h\p$ version of the method.
We denote with a), b) and c) the meshes whose elements are as in Figure \ref{figure possible geometric decompositions} (left), (center) and (right), respectively.
The geometric refinement parameters are $\sigma=\frac{1}{2}$ (left), $\sigma=\sqrt 2-1$ (center), $\sigma=(\sqrt 2-1)^2$ (right).
On each element, the local degree of accuracy is uniform and equal to the number of layers.
We depict for mesh a) the slope $\exp(-b\sqrt[2]{N})$ for some positive constant $b$, in order to check the exponential decay of the $H^1$ error.} \label{figure numerical test p uniform}
\end{figure}

\begin{figure}  [h]
\centering
\subfigure {\includegraphics [angle=0, width=0.32\textwidth]{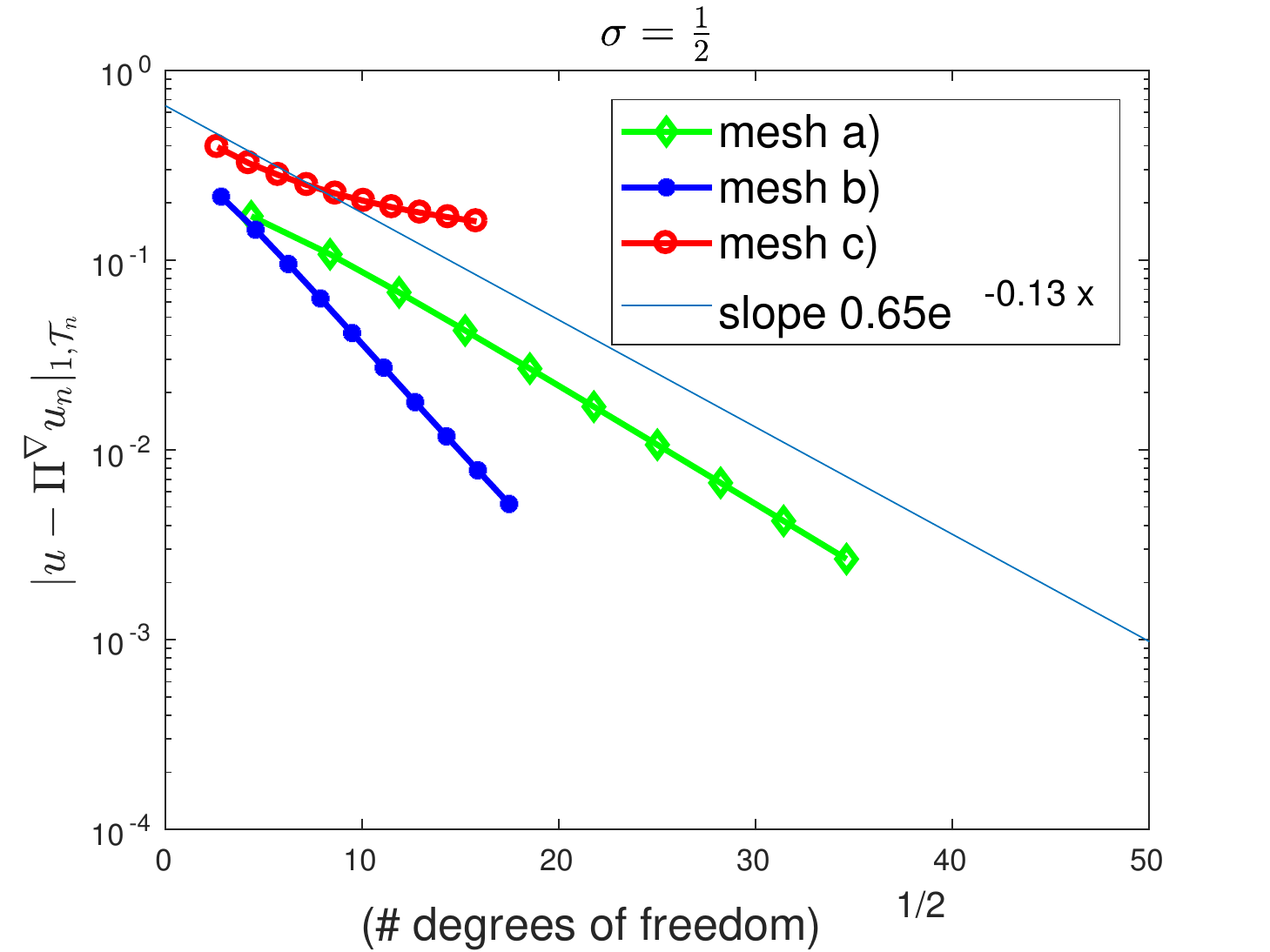}}
\subfigure {\includegraphics [angle=0, width=0.32\textwidth]{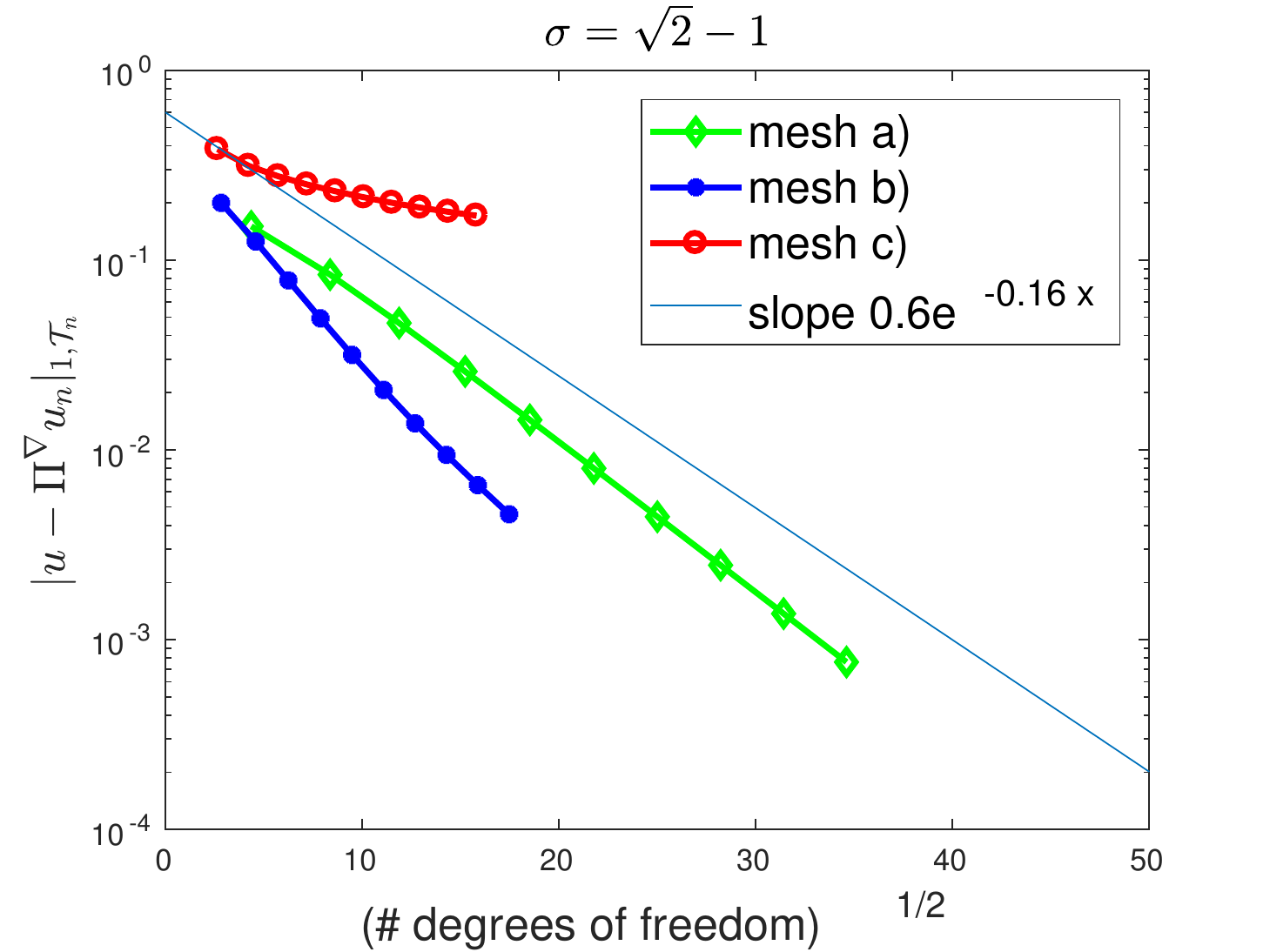}}
\subfigure {\includegraphics [angle=0, width=0.32\textwidth]{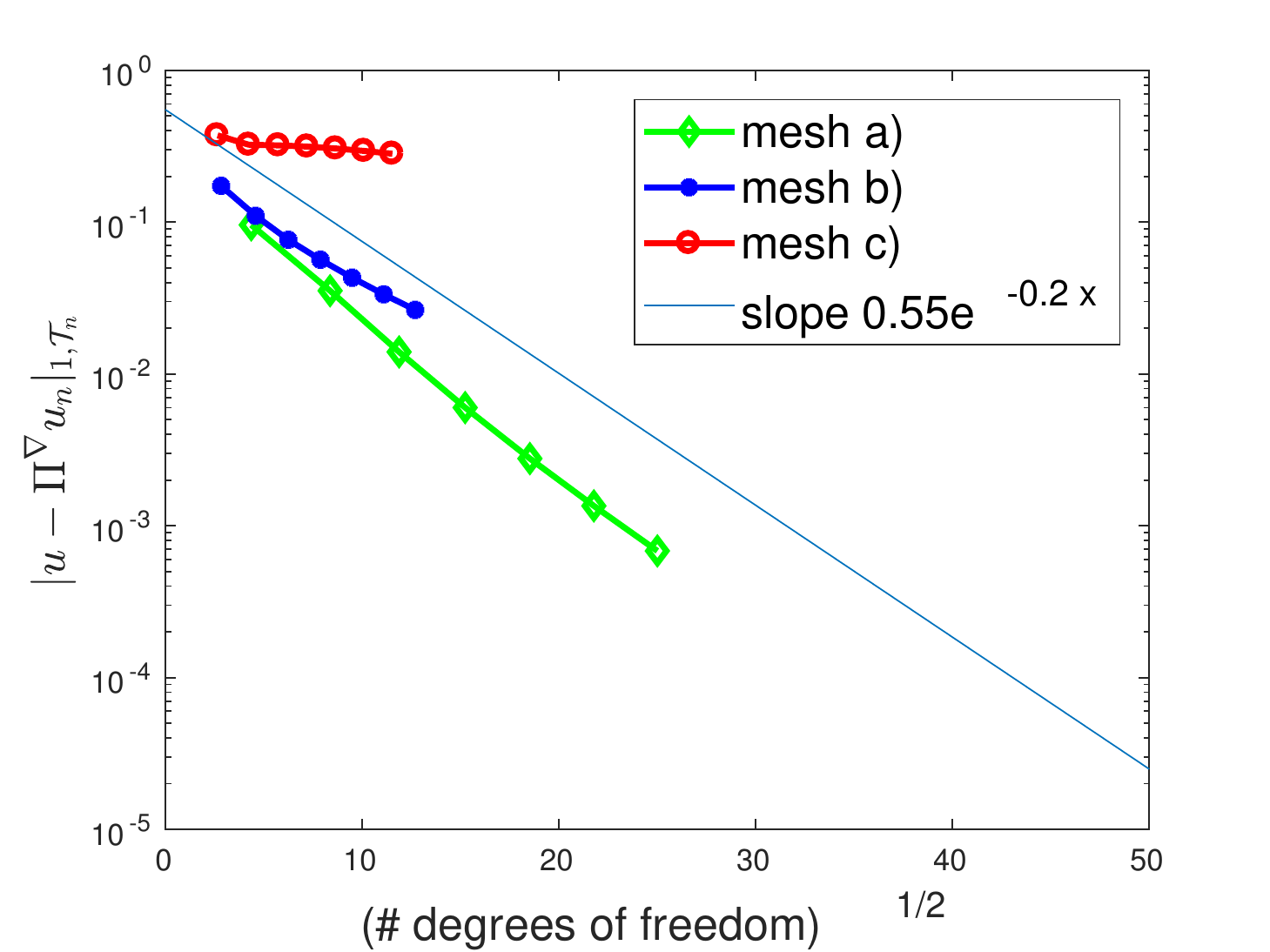}}
\caption{Error \eqref{computable error} on the three sequences of graded meshes made of elements as those in Figure \ref{figure possible geometric decompositions}.
We consider here the $\h\p$ version of the method.
We denote with a), b) and c) the meshes whose elements are as in Figure \ref{figure possible geometric decompositions} (left), (center) and (right), respectively.
The geometric refinement parameters are $\sigma=\frac{1}{2}$ (left), $\sigma=\sqrt 2-1$ (center), $\sigma=(\sqrt 2-1)^2$ (right).
The vector of local degrees of accuracy is nonuniform and given by $\p_\E = j+1$. $j=0,\dots,n$, $n+1$ being the number of layers in $\taun$.
We depict for mesh a) the slope $\exp(-b\sqrt[2]{N})$ for some positive constant $b$, in order to check the exponential decay of the $H^1$ error.} \label{figure numerical test p mu graded}
\end{figure}
As already stated in Example \ref{example meshes}, the meshes as those in Figure~\ref{figure possible geometric decompositions} (right) do not satisy the assumption~(\textbf{D1}) and then, in principle, Theorem \ref{theorem exponential convergence} does not apply.
The numerical experiments in Figure~\ref{figure numerical test p uniform} and~\ref{figure numerical test p mu graded} reveal that in this case the convergence deteriorates after few $hp$ refinements, especially for small~$\sigma$.

On the other hand, the other two sequences of meshes, namely those whose elements are depicted in Figure \ref{figure possible geometric decompositions} (left) and (center), have the expected exponential decay.

Importantly, the exponential convergence is still observed also under choice \eqref{graded degrees of accuracy} of the local degrees of accuracy. Our conjecture is that Theorem \ref{theorem exponential convergence} holds under \eqref{graded degrees of accuracy} as well.
Nonetheless, we avoid to investigate this issue, on the one hand, in order to avoid additional technicalities, on the other, because the dimension of space $\Vn$ under choices \eqref{standard choice degree of accuracy} and \eqref{graded degrees of accuracy}
behaves like $n^2$ and $\frac{1}{2}n^2$, respectively. This means that the exponential decay is still valid with the same exponential rate in both cases.

\subsubsection{Numerical comparison between $\h\p$ harmonic VEM and $\h\p$ VEM} \label{subsubsection comparison between HVEM and VEM}
We also perform a numerical comparison between the performances of the harmonic VEM discussed so far and the standard $hp$ version of VEM for the approximation of corner singularities, see \cite{hpVEMcorner}.
The main difference is that in VEM internal degrees of freedom for each element are employed in order to take care of the approximation of the right-hand side in Poisson problems.
This obviously leads to a nontrivial growth of the dimension of the space of approximation and in particular it leads to a decay of the energy error of the following sort:
\begin{equation} \label{decay error hpVEM}
\vert \u - \un \vert_{1,\O} \lesssim \exp{(-b \sqrt [3]{N})},
\end{equation}
where $b$ is a positive constant independent of the discretization parameters and $N$ is the dimension of the virtual space; see \cite[Theorem 3]{hpVEMcorner}.

In Figure \ref{figure comparison}, we compare error \eqref{computable error} for the two methods
employing the meshes in Figure \ref{figure possible geometric decompositions} (left) and in Figure \ref{figure possible geometric decompositions} (center).
The grading parameter is $\sigma = \frac{1}{2}$.

In both cases, we consider a distribution of local degrees of accuracy as in \eqref{standard choice degree of accuracy}. We note that the stabilization of the VEM differs from the one introduced in \eqref{our choice stabilization} for the harmonic VEM.
For more details concerning the construction of the $hp$ VEM we refer to \cite{hpVEMcorner}.

\begin{figure}  [h]
\centering
\subfigure {\includegraphics [angle=0, width=0.49\textwidth]{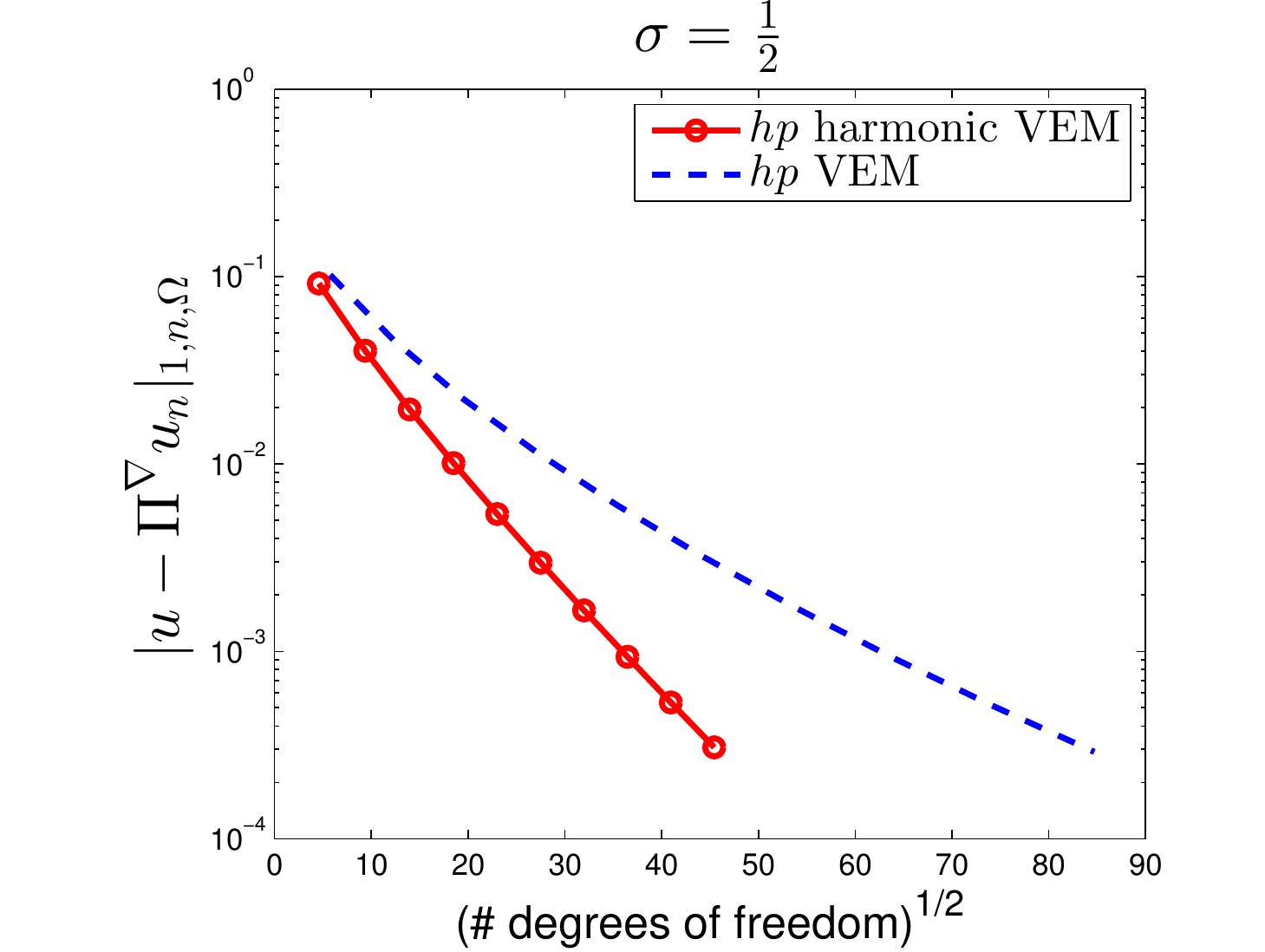}}
\subfigure {\includegraphics [angle=0, width=0.49\textwidth]{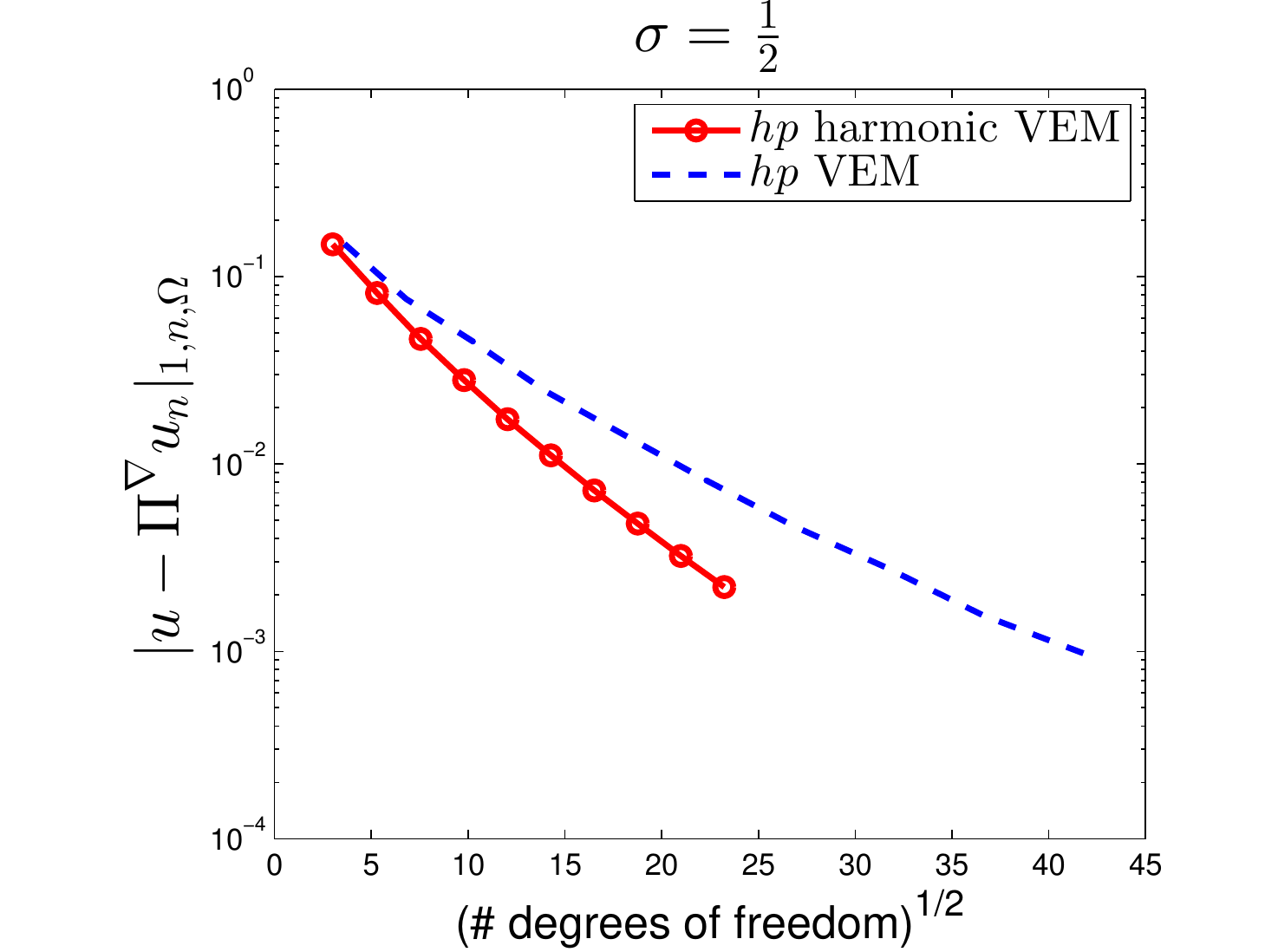}}
\caption{Comparison between the harmonic VEM and the VEM. Uniform degree of accuracy $\p=n+1$, $n+1$ being the number of layers.
We depict the error \eqref{computable error}. The geometric refinement parameters is $\sigma=\frac{1}{2}$.
Left: mesh in Figure \ref{figure possible geometric decompositions} (left). Right: mesh in Figure \ref{figure possible geometric decompositions} (center).} \label{figure comparison}
\end{figure}
From Figure \ref{figure comparison}, it is possible to observe the faster decay of error \eqref{computable error} when employing the $\h\p$ harmonic VEM,
see \eqref{error decay hpHVEM}, when compared to the same error employing the $\h\p$ VEM, see \eqref{decay error hpVEM}.

\section*{Aknowledgements}
L. M. acknowledges the support of the Austrian Science Fund (FWF) project F 65.
Both authors aknowledge that the major part of the research presented in this paper has been carried out at the Institute of Mathematics of the University of Oldenburg, Germany.
Moreover, they wish to thank Prof. M. J. Melenk of the Technische Universit\"at Wien for fruitful discussions on the topic.

{\footnotesize
\bibliography{bibliogr}
}
\bibliographystyle{plain}
\end{document}